\documentclass[12pt]{article}

\usepackage{amsmath,amssymb,amsthm,amsfonts,bbm,bm,mathtools,tikz}

\usetikzlibrary{calc,angles,quotes}

\usetikzlibrary{external}
\ifdefined\shellescape
\fi

\newtheorem{superclass}{Superclass}
\newtheorem{definition}[superclass]{Definition}
\newtheorem{lemma}[superclass]{Lemma}
\newtheorem{proposition}[superclass]{Proposition}
\newtheorem{theorem}[superclass]{Theorem}
\newtheorem{corollary}[superclass]{Corollary}
\newtheorem{example}[superclass]{Example}

\newcommand{\R}{\mathbb{R}}
\newcommand{\N}{\mathbb{N}}
\newcommand{\mc}{\mathcal}
\newcommand{\eps}{\epsilon}

\DeclareMathOperator{\argmin}{argmin}
\DeclareMathOperator{\argmax}{argmax}
\DeclareMathOperator{\dist}{dist}
\DeclareMathOperator{\co}{conv}
\DeclareMathOperator{\cone}{cone}
\DeclareMathOperator{\po}{poly}
\DeclareMathOperator{\interior}{int}

\DeclareMathOperator{\ext}{ext}

\begin{document}

\title{Galerkin approximations to the space of convex bodies by polytopes in nondegenerate
V-representation}
\author{Janosch Rieger}
\date{\today}
\maketitle

\begin{abstract}
We introduce finite-dimensional approximations to the space of convex bodies based on
polytopes in vertex representation.
For a prescribed set of directions, the admissible point configurations form a
polyhedral convex subcone of the Euclidean vector space, described by a finite
system of linear inequalities.
The interior of this cone contains only nondegenerate representations, in which all
parameter points are distinct vertices of the represented polytope.
We study the geometry of the parameter cone, redundancies in its defining inequalities,
and natural projections of convex bodies onto the resulting polytope spaces.
We derive quantitative approximation estimates in terms of how densely the prescribed
directions cover the unit sphere and construct nested Galerkin sequences whose
approximation error converges locally uniformly to zero.
Finally, we use these approximation properties to establish the convergence of
finite-dimensional approximations of constrained global optimization problems
in the space of convex bodies.
\end{abstract}


\medskip


\section{Introduction}

Optimization problems in the space of convex bodies have a long history.
Mahler's conjecture \cite{Mahler},
suggesting that the minimum of the product of the volumes of a
centrally symmetric convex body and its polar is attained at the cube, has been proved
in \cite{Mahler:2} and \cite{Iriyeh} for dimensions $2$ and $3$,
but is still open for all dimensions $d>3$.
Another long-standing open problem is to determine convex bodies of constant width
with minimal volume, where the Meissner bodies \cite{Kawohl} are conjectured
to be optimal for dimension $d=3$.
So far, only results for certain subfamilies of the convex bodies are known,
e.g.\ for the family of all bodies of revolution of constant width \cite{Anciaux}
and the family of all Meissner pyramids \cite{Bogosel:Meissner}.

\medskip

More recently, optimization problems in the space of convex bodies
have been approached from a computational perspective, partly with the
intention to understand problems like the above-mentioned ones better.

Approaches based on spectral decompositions of the support function
currently focus on dimensions $d=2,3$.
In \cite{Ftouhi}, for $d=2$, support and gauge functions are expressed in terms of
truncated Fourier series with convexity constraints imposed on the coefficients.
In \cite{Antunes} the same approach is pursued and extended to $d=3$.
The paper mentions that checking the convexity constraints at grid points
is not strictly sufficient to guarantee convexity, i.e.\
there is a chance that the represented function is not really a support function
of a convex body.
This problem had been solved before in \cite{Bayen}, but at the price of converting
the optimization problems to semidefinite programs.
In \cite{Bogosel} and \cite{Bogosel:24} the convexity constraint is made rigorous
by working with a finite difference-based approach.
It is not explicitly stated, but the proof of \cite[Theorem 2.6]{Bogosel} reveals
that the proposed spaces of convex bodies approximate the space of all convex bodies
in $\R^2$ locally uniformly.

Some work is based on the approximation by polytopes in the so-called $\mc{H}$-representation,
i.e.\ polytopes given as intersections of half-spaces.
It was already observed in \cite{Lachland} that this approach
\emph{yields a technical difficulty that should not be
underestimated: some of the boundary planes are \emph{dormant}, meaning the polytope
is actually included in the interior of the half-space.}
Therefore, in \cite{Diaz} the $\mc{H}$-representation is used locally, i.e.\ in a neighborhood
of a given parameter vector, where the polytopes do not change their combinatorial
structure.
To resolve the issue, the author developed in \cite{Rieger:Galerkin} a framework
that yields a convex polyhedral subcone of $\R^N$ containing all nondegenerate
representations of polytopes with $N$ prescribed facet normals in $\R^d$
for arbitrary $d$.
These spaces can be chosen in such a way that they approximate the space of convex
bodies in $\R^d$ locally uniformly to arbitrary precision.
This approach has been applied in \cite{Ernst} for approximating the infinite time
reachable set of strictly stable linear control systems, which is a convex body in $\R^d$
that is difficult to characterize.

Naturally, some work is based on the approximation of convex bodies by polytopes
in $\mc{V}$-representation, i.e.\ by polytopes that are given in terms of their vertices.
The issue here is that $\mc{V}$-representations degenerate when points
that are meant to be vertices of the parameterized polytope lie in the convex hull
of the remaining points.
In \cite{Bartels:2} and \cite{Bartels} this issue is circumvented by meshing
the body and keeping the mesh structure fixed, which resembles the approach
by \cite{Diaz} in the dual setting.
Once the facets are known and fixed, it suffices to ensure that all pairs
of neighboring facets share a dihedral angle no larger than $\pi$.

The above overview is not exhaustive, and the attempt to systematically approximate
convex objects is not limited to convex bodies.
We mention that e.g.\ in \cite{Ekeland} convex approximations to convex functions
are constructed by considering function values and gradients on grid points
as optimization variables and enforcing the standard convexity criterion
for functions in terms of these function values and gradients.

\medskip

The purpose of this paper is to propose a solution to the issue of degeneracy
of approximations in $\mc{V}$-representation that is not limited to a particular dimension
or a fixed combinatorial structure.
We essentially consider the space of all polytopes in $\R^d$ with prescribed vertex normals.
These polytopes are precisely the convex hulls of point configurations that satisfy a system
of linear inequalities, i.e.\ the set of all admissible point configurations forms
a polyhedral convex cone.
In Section \ref{sec:background} we introduce notation, recall our work from
\cite{Rieger:Galerkin} and explain how the spaces introduced in this paper
complement the use of those from \cite{Rieger:Galerkin}.
In Section \ref{sec:the:cone} we explore the geometry of the cone of admissible parameters
and the geometry of the parameterized polytopes.
In Section \ref{sec:interior} we prove that the interior of the parameter cone
contains only point configurations such that all points are distinct vertices of the
parameterized polytopes.
In Section \ref{sec:redundancy} we briefly discuss redundancies in the linear
inequality constraints defining the parameter cones.
In Section \ref{sec:projectors} we explore properties of natural projectors from
the space of convex bodies to the parameter cone and to the space of parameterized
polytopes.
In Section \ref{sec:approximation} we prove that our polytope spaces approximate
the space of convex bodies locally uniformly, and that the locally uniform
approximation error tends to zero when the chordal covering radius of the
vertex normals in the sphere converges to zero.
Finally, in Section \ref{sec:optimization}, we lay out how these spaces
are used to approach global optimization problems in the space of convex bodies.

\section{Background and definitions}
\label{sec:background}

Let $d\ge 2$.
By $\mc{C}^d$, $\mc{K}^d$ and $\mc{CC}^d$ we denote the spaces of
all nonempty and compact,
all nonempty, convex and compact,
and all nonempty, closed and convex subsets of $\R^d$.
The Hausdorff semi-distance and Hausdorff distance are, respectively,
\begin{align*}
&\dist:\mc{C}^d\times\mc{C}^d\to\R_+,&&\dist(K,L):=\sup_{x\in K}\inf_{y\in L}\|x-y\|,\\
&\dist_H:\mc{C}^d\times\mc{C}^d\to\R_+,&&\dist_H(K,L):=\max\{\dist(K,L),\dist(L,K)\}.
\end{align*}
Similarly, we introduce the size
\[\|\,\cdot\,\|:\mc{C}^d\to\R_+,\quad\|K\|:=\max_{x\in K}\|x\|.\]
We denote the set of all extremal points of $K\in\mc{K}^d$ by $\ext(K)$
and the convex hull of $K\in\mc{C}^d$ by $\co(K)$.
We often drop the brackets for aesthetic reasons.
Also, when a set is a singleton or a collection consists of a single set,
we may state the singleton instead of the set or the set instead of the collection.
The normal cone to a set $K\in\mc{K}^d$ at $x\in K$ is the set
\[\mc{N}_K(x):=\{z\in\R^d:z^T(y-x)\le0\ \text{for all}\ y\in K\}.\]

Throughout the first sections of this paper we fix $N\in\N$ with $N>1$
and pairwise distinct vectors
$a_1,\ldots,a_N\in\R^d$ with $\|a_i\|=1$ for all $i\in\{1,\ldots,N\}$
and refer to them in the compact form $\bm{a}\in\R^{d\times N}$.

\subsection{Overapproximation in $\mc{H}$-representation}

Given $\bm{a}\in\R^{d\times N}$, the paper \cite{Rieger:Galerkin} considers the space
\[\mc{P}^\sharp_{\bm{a}}:=\{\po_{\bm{a}}(b):b\in\R^N\}\setminus\{\emptyset\},\quad
\po_{\bm{a}}(b):=\{x\in\R^d:\bm{a}^Tx\le b\},\]
of all polyhedra in $\R^d$ the facet normals of which are contained in $\{a_1,\ldots,a_N\}$.
If the chordal covering radius of $\bm{a}$ in the sphere is not too large,
all these polyhedra are polytopes.
The space $\mc{P}^\sharp_{\bm{a}}$ is parameterized by the set
\[\mc{C}^\sharp_{\bm{a}}
:=\{b\in\R^N:\forall\,i\in\{1,\ldots,N\}\
\exists x_i\in\po_{\bm{a}}(b)\ \text{with}\ a_i^Tx_i = b_i\}\]
of nondegenerate parameterizations and Sections 2.4-2.6 establish
that $\mc{C}^\sharp_{\bm{a}}$ is a polyhedral convex cone
with an $\mc{H}$-representation that can be derived from $\bm{a}$.
While the representation of $\po_{\bm{a}}(b)$ in terms of $b$ is sparse, the system of
linear inequalities that define $\mc{C}^\sharp_{\bm{a}}$ is large and somewhat
difficult to work with.
The space is well-suited for projects like \cite{Ernst}
that involve overapproximation, ordering by set
inclusion, which is preserved by the projectors
\begin{align}
&\pi^\sharp_{\bm{a}}:\mc{K}^d\rightarrow\mc{C}^\sharp_{\bm{a}},
&&[\pi^\sharp_{\bm{a}}(K)]_i:=\max_{x\in K}a_i^Tx\quad\text{for all}\ i,\label{sharp:1}\\
&\Pi^\sharp_{\bm{a}}:\mc{K}^d\rightarrow\mc{P}^\sharp_{\bm{a}},
&&\Pi^\sharp_{\bm{a}}(K):=\po_{\bm{a}}(\pi^\sharp_{\bm{a}}(K)),\label{sharp:2}
\end{align}
or problems that involve the width of convex bodies.
In addition, the space has good approximation properties, which we restate from
\cite[Theorem 38]{Rieger:Galerkin}.

\begin{theorem} \label{oldapprox}
Let the chordal covering radius
\[\delta_{\bm{a}}:=\max_{u\in S^{d-1}}\min_{i\in\{1,\ldots,N\}}\|u-a_i\|\]
of $\bm{a}$ satisfy $\delta_{\bm{a}}\in(0,1)$.
Then for every $K\in\mc{K}^d$, we have $K\subset\Pi^\sharp_{\bm{a}}(K)$ and
\[\dist(\Pi^\sharp_{\bm{a}}(K),K)
\le\frac{2-\delta_{\bm{a}}}{1-\delta_{\bm{a}}}\delta_{\bm{a}}\|K\|.\]
\end{theorem}

However, the approach is not suitable for underapproximation and
applications in which the explicit knowledge of the vertices is required,
e.g.\ for mesh generation like in \cite{Bartels:2} and \cite{Bartels}.

\subsection{Underapproximation in $\mc{V}$-representation}

The following approach pursued in this paper intends to meet the above-mentioned needs.
We denote
\[\bm{x}:=(x_1,\ldots,x_N)\in\R^{d\times N}
\quad\text{and}\quad
\{\bm{x}\}:=\{x_1,\ldots,x_N\}\subset\R^d\]
and consider the spaces
\begin{align*}
&\mc{C}^\flat_{\bm{a}}:=\{\bm{x}\in\R^{d\times N}:
a_i^Tx_j\le a_i^Tx_i\ \text{when}\ i\neq j\},\\
&\mc{P}^\flat_{\bm{a}}:=\{\co\{\bm{x}\}:\bm{x}\in\mc{C}^\flat_{\bm{a}}\}.
\end{align*}
For $\bm{x}\in\mc{C}^\flat_{\bm{a}}$, the set $\co\{\bm{x}\}$ is a polytope in $\mc{V}$-representation,
and hence $\mc{P}^\flat_{\bm{a}}\subset\mc{K}^d$.
Figure \ref{figure:def:spaces} illustrates how $\bm{a}$ determines $\mc{P}^\flat_{\bm{a}}$.
We endow $\mc{P}^\flat_{\bm{a}}$ with $\dist_H$ and $\mc{C}^\flat_{\bm{a}}$ with the distance
\[\dist_{\mc{C}^\flat_{\bm{a}}}(\bm{x},\bm{y})
:=\max_{i\in\{1,\ldots,N\}}\|x_i-y_i\|.\]

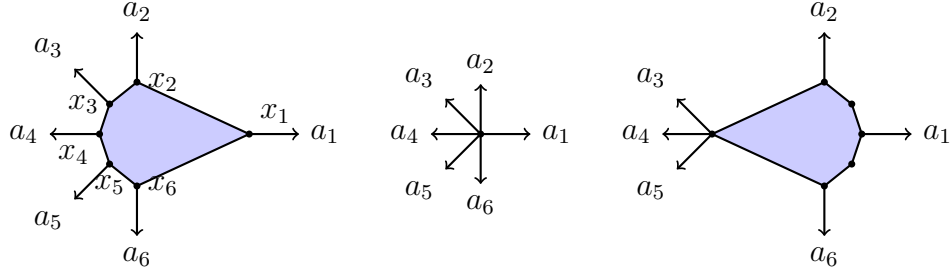
\begin{figure}
\centering
\begin{tabular}{lcr}

\begin{tikzpicture}[scale=0.66]
\useasboundingbox (-2,-2.5) rectangle (3,2.5);
\coordinate (L) at (1.5,0);
\foreach \i/\ang in {R1/-60,R2/-30,R3/0,R4/30,R5/60}
  {\coordinate (\i) at ({-1.5*cos(\ang)},{1.2*sin(\ang)});}
\fill[blue!20] (L) -- (R5) -- (R4) -- (R3) -- (R2) -- (R1) -- cycle;
\draw[thick] (L) -- (R5) -- (R4) -- (R3) -- (R2) -- (R1) -- (L);

\fill (L) circle (2pt) node[above right] {$x_1$};
\fill (R5) circle (2pt) node[right] {$x_2$};
\fill (R4) circle (2pt) node[left] {$x_3$};
\fill (R3) circle (2pt) node[below left] {$x_4$};
\fill (R2) circle (2pt) node[below] {$x_5$};
\fill (R1) circle (2pt) node[right] {$x_6$};

\draw[->,thick] (L) -- ($ (L) + (1,0) $) node[right] {$a_1$};
\draw[->,thick] (R1) -- ($ (R1) + (0,-1) $) node[below] {$a_6$};
\draw[->,thick] (R2) -- ($ (R2) + (-0.707,-0.707) $) node[below left] {$a_5$};
\draw[->,thick] (R3) -- ($ (R3) + (-1,0) $) node[left] {$a_4$};
\draw[->,thick] (R4) -- ($ (R4) + (-0.707,0.707) $) node[above left] {$a_3$};
\draw[->,thick] (R5) -- ($ (R5) + (0,1) $) node[above] {$a_2$};
\end{tikzpicture}

&

\begin{tikzpicture}[scale=0.66]
\useasboundingbox (-2.5,-2.5) rectangle (2.5,2.5);
\fill (0,0) circle (2pt);
\draw[->,thick] (0,0) -- (1,0) node[right] {$a_1$};
\draw[->,thick] (0,0) -- (0,1) node[above] {$a_2$};
\draw[->,thick] (0,0) -- (-0.707, 0.707) node[above left] {$a_3$};
\draw[->,thick] (0,0) -- (-1,0) node[left] {$a_4$};
\draw[->,thick] (0,0) -- (-0.707,-0.707) node[below left] {$a_5$};
\draw[->,thick] (0,0) -- (0,-1) node[below] {$a_6$};
\end{tikzpicture}

&

\begin{tikzpicture}[scale=0.66]
\useasboundingbox (-3,-2.5) rectangle (2,2.5);
\coordinate (L) at (-1.5,0);
\foreach \i/\ang in {R1/-60,R2/-30,R3/0,R4/30,R5/60}
  {\coordinate(\i)at({1.5*cos(\ang)},{1.2*sin(\ang)});}
\fill[blue!20] (L) -- (R5) -- (R4) -- (R3) -- (R2) -- (R1) -- cycle;
\draw[thick](L)--(R1)--(R2)--(R3)--(R4)--(R5)--(L);
\fill (L) circle (2pt);
\foreach\v in{R1,R2,R3,R4,R5}{\fill (\v) circle (2pt);}
\draw[->,thick] (L) -- ($ (L) + (-1,0) $) node[left] {$a_4$};
\draw[->,thick] (R1) -- ($ (R1) + (0,-1) $) node[below] {$a_6$};
\draw[->,thick] (R3) -- ($ (R3) + (1,0) $) node[right] {$a_1$};
\draw[->,thick] (R5) -- ($ (R5) + (0,1) $) node[above] {$a_2$};
\draw[->,thick] (L) -- ($ (L) + (-0.707,0.707) $) node[above left] {$a_3$};
\draw[->,thick] (L) -- ($ (L) + (-0.707,-0.707) $) node[below left] {$a_5$};
\end{tikzpicture}

\end{tabular}
\caption{The polytope on the left is an element of $\mc{P}^\flat_{\bm{a}}$ for $\bm{a}$
shown in the middle, but the polytope on the right is not.\label{figure:def:spaces}}
\end{figure}

We define multivalued projectors
$\pi^\flat_{\bm{a}}:\mc{K}^d\rightrightarrows\mc{C}^\flat_{\bm{a}}$
and $\Pi^\flat_{\bm{a}}:\mc{K}^d\rightrightarrows\mc{P}^\flat_{\bm{a}}$ by
\begin{align}
&\pi^\flat_{\bm{a}}(K):=\{\bm{x}\in\R^{d\times N}:x_i\in\argmax_{x\in K}a_i^Tx\ \text{for all}\ i\},\label{def:pi}\\
&\Pi^\flat_{\bm{a}}(K):=\{\co\{\bm{x}\}:\bm{x}\in\pi^\flat_{\bm{a}}(K)\}.\label{def:Pi}
\end{align}

The projectors $\pi^\flat_{\bm{a}}$ and $\Pi^\flat_{\bm{a}}$ are indeed multivalued.

\begin{example}\label{example:multivalued:projector}
Let $d=2$, let $K:=B_1(0)\cap\{x\in\R^2:x_2\ge0\}$, and let
$a_1=\frac{1}{\sqrt{2}}(1,1)^T$, $a_2=(0,1)^T$, $a_3=\frac{1}{\sqrt{2}}(-1,1)^T$ and $a_4=(0,-1)^T$.
Then
\begin{align*}
&\pi^\flat_{\bm{a}}(K)=\{(\tfrac{1}{\sqrt{2}}(1,1)^T,(0,1)^T,\tfrac{1}{\sqrt{2}}(-1,1)^T,(s,0)^T):s\in[-1,1]\},\\
&\Pi^\flat_{\bm{a}}(K)=\{\co\{\tfrac{1}{\sqrt{2}}(1,1)^T,(0,1)^T,\tfrac{1}{\sqrt{2}}(-1,1)^T,(s,0)^T\}:s\in[-1,1]\}.
\end{align*}
This situation is illustrated in Figure \ref{figure:example:multivalued:projector}.
\end{example}

\begin{figure}
\centering
\begin{tabular}{lcr}

\begin{tikzpicture}[scale=1.3]
    \useasboundingbox (-1.2,-1.2) rectangle (1.2,1.2);
    \coordinate (A1) at ({1/sqrt(2)}, {1/sqrt(2)});
    \coordinate (A2) at (0,1);
    \coordinate (A3) at ({-1/sqrt(2)}, {1/sqrt(2)});
    \coordinate (S) at (-0.5,0);
    \coordinate (O) at (0,0);

    \fill[blue!20] (A1) -- (A2) -- (A3) -- (S) -- cycle;
    \draw[thick] (A1) -- (A2) -- (A3) -- (S) -- cycle;
    \draw[thick] (0,0) -- (180:1) arc (180:0:1) -- cycle;
    \draw[thick,dashed] (-1.35,0) -- (-1,0);
    \draw[thick,dashed] (1,0) -- (1.39,0);
    \draw[thick,dashed] (1,0) arc (0:-180:1);

    \fill (A1) circle (1pt) node[right] {$x_1$};
    \fill (A2) circle (1pt) node[above] {$x_2$};
    \fill (A3) circle (1pt) node[left] {$x_3$};
    \fill (S) circle (1pt) node[below] {$x_4$};
    \fill (O) circle (1pt) node[below] {$0$};
    \node at (1.2,0.35) {$K$};
\end{tikzpicture}

&

\begin{tikzpicture}[scale=1.3]
    \useasboundingbox (-1.2,-1.2) rectangle (1.2,1.2);
    \coordinate (a1) at ({1/sqrt(2)}, {1/sqrt(2)});
    \coordinate (a2) at (0,1);
    \coordinate (a3) at ({-1/sqrt(2)}, {1/sqrt(2)});
    \coordinate (a4) at (0,-1);
    \coordinate (O) at (0,0);

    \fill (O) circle (1pt) node[below right] {$0$};
    \draw[->,thick] (0,0) -- (a1) node[above right] {$a_1$};
    \draw[->,thick] (0,0) -- (a2) node[above] {$a_2$};
    \draw[->,thick] (0,0) -- (a3) node[above left] {$a_3$};
    \draw[->,thick] (0,0) -- (a4) node[below] {$a_4$};
\end{tikzpicture}

&

\begin{tikzpicture}[scale=1.3]
    \useasboundingbox (-1.2,-1.2) rectangle (1.2,1.2);
    \coordinate (A1) at ({1/sqrt(2)}, {1/sqrt(2)});
    \coordinate (A2) at (0,1);
    \coordinate (A3) at ({-1/sqrt(2)}, {1/sqrt(2)});
    \coordinate (S) at (0.5,0);
    \coordinate (O) at (0,0);

    \fill[blue!20] (A1) -- (A2) -- (A3) -- (S) -- cycle;
    \draw[thick] (A1) -- (A2) -- (A3) -- (S) -- cycle;
    \draw[thick] (0,0) -- (180:1) arc (180:0:1) -- cycle;
    \draw[thick,dashed] (-1.35,0) -- (-1,0);
    \draw[thick,dashed] (1,0) -- (1.39,0);
    \draw[thick,dashed] (1,0) arc (0:-180:1);

    \fill (A1) circle (1pt) node[right] {$x_1$};
    \fill (A2) circle (1pt) node[above] {$x_2$};
    \fill (A3) circle (1pt) node[left] {$x_3$};
    \fill (S) circle (1pt) node[below] {$x_4$};
    \fill (O) circle (1pt) node[below] {$0$};
    \node at (1.2,0.35) {$K$};
\end{tikzpicture}

\end{tabular}
\caption{Two elements of $\pi^\flat_{\bm{a}}(K)$ with $s=-1/2$ (left) and $s=1/2$ (right)
for $K$ and $\bm{a}$ (middle) from Example \ref{example:multivalued:projector}.
The shaded polytopes are the corresponding elements of the collection $\Pi^\flat_{\bm{a}}(K)$.
\label{figure:example:multivalued:projector}}
\end{figure}
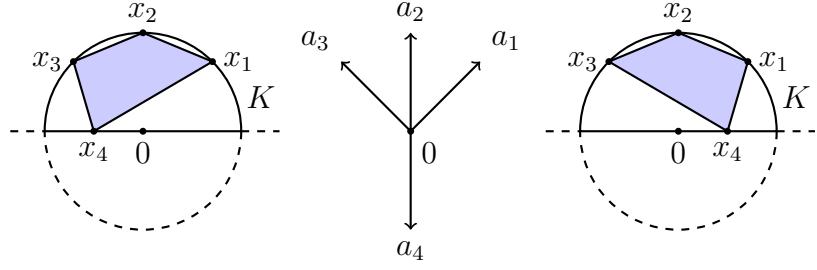

\section{The cone $\mc{C}^\flat_{\bm{a}}$}
\label{sec:the:cone}

We commence with a geometric interpretation of $\bm{x}\in\mc{C}^\flat_{\bm{a}}$,
which implies that $x_i\notin\interior(\co\bm{x})$ for all $i$.

\begin{lemma}\label{lem:normal:cone:relationship}
Let $\bm{x}\in\R^{d\times N}$.
Then $\bm{x}\in\mc{C}^\flat_{\bm{a}}$ if and only if
$a_i\in\mc{N}_{\co\{\bm{x}\}}(x_i)$ for all $i$.
\end{lemma}

\begin{proof}
If $a_i\in\mc{N}_{\co\{\bm{x}\}}(x_i)$ for all $i$,
then by the definition of the normal cone $a_i^T(x_j-x_i)\le0$ for $i\neq j$,
which is the very definition of $\bm{x}\in\mc{C}^\flat_{\bm{a}}$.

Let $\bm{x}\in\mc{C}^\flat_{\bm{a}}$, fix $i$, and let $x\in\co\{\bm{x}\}$.
There exists $\lambda\in\R^N_+$ with $\mathbbm{1}^T\lambda=1$ and
$\sum_{j=1}^N\lambda_jx_j=x$, and we have
$a_i^T(x-x_i)=\sum_{j=1}^N\lambda_ja_i^T(x_j-x_i)\le0$
by definition of $\mc{C}^\flat_{\bm{a}}$.
Hence $a_i\in\mc{N}_{\co\{\bm{x}\}}(x_i)$.
\end{proof}

The following lemma follows directly from the definition of $\mc{C}^\flat_{\bm{a}}$.

\begin{lemma}\label{lemma:polyhedral:convex:cone}
The space $\mc{C}^\flat_{\bm{a}}\subset\R^{d\times N}$ is a polyhedral closed convex cone.
\end{lemma}

The number of extremal points of polytopes in $\mc{P}_{\bm{a}}^\flat$
cannot exceed $N$.

\begin{lemma}\label{lem:at:most:N:vertices}
If $\bm{x}\in\mc{C}^\flat_{\bm{a}}$, then $\#\ext\co\{\bm{x}\}\le N$.
\end{lemma}

\begin{proof}
By \cite[Proposition 2.2]{Ziegler} we have $\ext(\co\{\bm{x}\})\subset\{\bm{x}\}$.
\end{proof}

The Hausdorff semi-distance of two arbitrary polytopes is bounded by the Hausdorff
semi-distance between the sets of points defining them.

\begin{lemma}\label{lemma:vertices:distance}
For $\bm{x},\bm{y}\in\R^{d\times N}$ we have
\[\dist(\co\{\bm{x}\},\co\{\bm{y}\})\le\dist(\{\bm{x}\},\{\bm{y}\}).\]
\end{lemma}

\begin{proof}
Let $x\in\co\{\bm{x}\}$.
Then there exists $\lambda\in\R^N_+$ such that $\mathbbm{1}^T\lambda=1$ and
$\sum_{i=1}^N\lambda_ix_i=x$.
Select $j_i\in\argmin_{j\in\{1,\ldots,N\}}\|x_i-y_j\|$ for all $i$.
Then we have $y:=\sum_{i=1}^N\lambda_iy_{j_i}\in\co\{\bm{y}\}$, and
\begin{align*}
\dist(x,\co\{\bm{y}\})
&\le\|x-y\|
=\|\sum_{i=1}^N\lambda_i(x_i-y_{j_i})\|\\
&\le\sum_{i=1}^N\lambda_i\|x_i-y_{j_i}\|
=\sum_{i=1}^N\lambda_i\dist(x_i,\{\bm{y}\})
\le\dist(\{\bm{x}\},\{\bm{y}\}).
\end{align*}
Since $x$ was arbitrary, the desired statement holds.
\end{proof}

This immediately implies the following result.

\begin{corollary}\label{cor:embedding:Lipschitz}
The mapping
$\co:(\mc{C}^\flat_{\bm{a}},\dist_{\mc{C}^\flat_{\bm{a}}})\to(\mc{P}^\flat_{\bm{a}},\dist_H)$
is $1$-Lipschitz.
\end{corollary}

\begin{proof}
For any $\bm{x},\bm{y}\in\mc{C}^\flat_{\bm{a}}$, we have
\[\dist(\co\{\bm{x}\},\co\{\bm{y}\})\le\dist(\{\bm{x}\},\{\bm{y}\})
\le\dist_{\mc{C}^\flat_{\bm{a}}}(\bm{x},\bm{y}).\]
The opposite inequality follows by reversing the roles of $\bm{x}$ and $\bm{y}$.
\end{proof}

It is tempting to think that for points $\bm{x},\bm{y}\in\mc{C}^\flat_{\bm{a}}$
Lemma \ref{lemma:vertices:distance} should simplify
to $\dist(\co\{\bm{x}\},\co\{\bm{y}\})=\dist_{\mc{C}^\flat_{\bm{a}}}(\bm{x},\bm{y})$,
because $\mc{C}^\flat_{\bm{a}}$ restricts how the $x_i$ and $y_i$ can be ordered in space.
This is false.

\begin{example}\label{counterexample:vertex:order}
Let $a_1=(1,0)^T$, $a_2=(0,1)^T$ and $a_3=(-1,0)^T$, and consider
$(x_1,x_2,x_3)\in\mc{C}^\flat_{\bm{a}}$ and $(y_1,y_2,y_3)\in\mc{C}^\flat_{\bm{a}}$
given by $x_1=(1,0)^T$, $x_2=(-\frac45,\frac15)^T$, $x_3=(-1,0)^T$
and $y_1=(1,0)^T$, $y_2=(\frac45,\frac15)^T$, $y_3=(-1,0)^T$.
We have
\begin{align*}
\dist_{\mc{C}^\flat_{\bm{a}}}(\bm{x},\bm{y})
&\ge\|x_2-y_2\|
=8/5>\sqrt{2}/5\\
&=\|x_2-y_3\|\ge\dist(x_2,\co\{\bm{y}\})
=\dist(\co\{\bm{x}\},\co\{\bm{y}\}).
\end{align*}
\end{example}

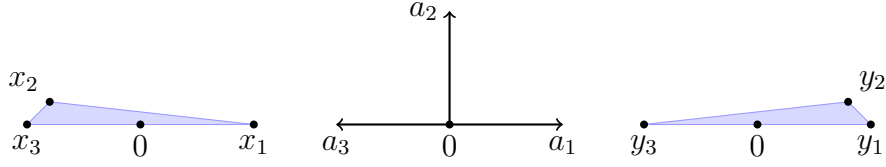
\begin{figure}
\centering
\begin{tabular}{lcr}

\begin{tikzpicture}[scale=1.5]
  \fill[blue!15] (1,0) -- (-0.8,0.2) -- (-1,0) -- cycle;
  \draw[blue!40] (1,0) -- (-0.8,0.2) -- (-1,0) -- cycle;
  \fill (1,0) circle (1pt) node [below] {$x_1$};
  \fill (-0.8,0.2) circle (1pt) node [above left] {$x_2$};
  \fill (-1,0) circle (1pt) node [below] {$x_3$};
  \coordinate (O) at (0,0);
  \fill (O) circle (1pt) node [below] {$0$};
\end{tikzpicture}

&

\begin{tikzpicture}[scale=1.5]
  \draw[->, thick] (0,0) -- (1,0) node[below] {$a_1$};
  \draw[->, thick] (0,0) -- (0,1) node[left] {$a_2$};
  \draw[->, thick] (0,0) -- (-1,0) node[below] {$a_3$};
  \coordinate (O) at (0,0);
  \fill (O) circle (1pt) node [below] {$0$};
\end{tikzpicture}

&

\begin{tikzpicture}[scale=1.5]
  \fill[blue!15] (1,0) -- (0.8,0.2) -- (-1,0) -- cycle;
  \draw[blue!40] (1,0) -- (0.8,0.2) -- (-1,0) -- cycle;
  \fill (1,0) circle (1pt) node[below] {$y_1$};
  \fill (0.8,0.2) circle (1pt) node[above right] {$y_2$};
  \fill (-1,0) circle (1pt) node[below] {$y_3$};
  \coordinate (O) at (0,0);
  \fill (O) circle (1pt) node [below] {$0$};
\end{tikzpicture}

\end{tabular}
\caption{In Example \ref{counterexample:vertex:order}, we have $(x_1,x_2,x_3)\in\mc{C}^\flat_{\bm{a}}$
(left) and $(y_1,y_2,y_3)\in\mc{C}^\flat_{\bm{a}}$ (right) with same $\bm{a}$ (middle), but
$\|x_2-y_2\|>\dist(x_2,\{y_1,y_2,y_3\})$.}
\label{figure:no:order}
\end{figure}

Strict inequalities in the definition of $\bm{x}\in\mc{C}^\flat_{\bm{a}}$ generate vertices
in $\co\{\bm{x}\}$.

\begin{lemma}\label{lem:strict:inequality:vertex}
Let $\bm{x}\in\mc{C}^\flat_{\bm{a}}$, and let $i$ be such that
$a_i^Tx_j<a_i^Tx_i$ for $i\neq j$.
Then $x_i\in\ext(\co\{\bm{x}\})$ 
and $x_i\neq x_j$ for $i\neq j$.
\end{lemma}

\begin{proof}
If $x_i\notin\ext(\co\{\bm{x}\})$
then there exists $\lambda\in\R^N_+$ with $\mathbbm{1}^T\lambda=1$, $\lambda\neq e_i$ and
$\sum_{j=1}^N\lambda_jx_j=x_i$.
But then $\lambda_j>0$ for at least one $j\neq i$, and we arrive at the contradiction
\[a_i^Tx_i=\sum_{j=1}^N\lambda_ja_i^Tx_j<\sum_{j=1}^N\lambda_ja_i^Tx_i=a_i^Tx_i.\]

If $x_i=x_j$ and $i\neq j$, we obtain the contradiction
$a_i^Tx_i=a_i^Tx_j<a_i^Tx_i$.
\end{proof}

The converse of Lemma \ref{lem:strict:inequality:vertex} is not true.

\begin{example}\label{borderline}
Let $a_1=(\sqrt3/2,-1/2)^T$, $a_2=(0,1)^T$, $a_3=(-\sqrt3/2,-1/2)^T$
and consider $\bm{x}\in\mc{C}^\flat_{\bm a}$ given by
$x_1=(1,1)^T$, $x_2=(-1,1)^T$, $x_3=(0,-1)^T$.
Then clearly $\{\bm{x}\}=\ext(\co\{\bm{x}\})$ and $x_i\neq x_j$ for $i\neq j$,
but $a_1^Tx_3=a_1^Tx_1$,
$a_2^Tx_1=a_2^Tx_2$ and
$a_3^Tx_2=a_3^Tx_3$.
This is illustrated in Figure \ref{fig:borderline}.
\end{example}

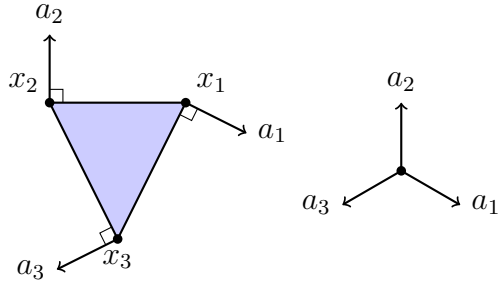
\begin{figure}[b]
\centering
\begin{tabular}{lr}

\begin{tikzpicture}[scale=0.9]
\useasboundingbox (-1.7,-1.7) rectangle (2,2);

\coordinate (X1) at (1,1);
\coordinate (X2) at (-1,1);
\coordinate (X3) at (0,-1);

\coordinate (A1) at ($(X1)+(0.894,-0.447)$);
\coordinate (A2) at ($(X2)+(0,1)$);
\coordinate (A3) at ($(X3)+(-0.894,-0.447)$);

\fill[blue!20] (X1)--(X2)--(X3)--cycle;
\draw[thick] (X1)--(X2)--(X3)--cycle;

\fill (X1) circle (2pt) node[above right] {$x_1$};
\fill (X2) circle (2pt) node[above left] {$x_2$};
\fill (X3) circle (2pt) node[below] {$x_3$};

\draw[->,thick] (X1) -- (A1) node[right] {$a_1$};
\draw[->,thick] (X2) -- (A2) node[above] {$a_2$};
\draw[->,thick] (X3) -- (A3) node[left] {$a_3$};

\pic[draw,angle radius=5pt] {right angle=X3--X1--A1};
\pic[draw,angle radius=5pt] {right angle=X1--X2--A2};
\pic[draw,angle radius=5pt] {right angle=X2--X3--A3};

\end{tikzpicture}

&

\begin{tikzpicture}[scale=0.9]
\useasboundingbox (-1.7,-1.7) rectangle (2,2);

\fill (0,0) circle (2pt);

\draw[->,thick] (0,0) -- (0.866,-0.5) node[right] {$a_1$};
\draw[->,thick] (0,0) -- (0,1) node[above] {$a_2$};
\draw[->,thick] (0,0) -- (-0.866,-0.5) node[left] {$a_3$};
\end{tikzpicture}
\end{tabular}
\caption{Configuration discussed in Example \ref{borderline}.
All points $x_i$ from the tuple $\bm{x}=(x_1,x_2,x_3)\in\mc{C}^\flat_{\bm{a}}$
are vertices despite several nonstrict inequalities.}
\label{fig:borderline}
\end{figure}

The space $\mc{P}^\flat_{\bm{a}}$ is \emph{flat} in $\mc{K}^d$.

\begin{theorem}\label{thm:P:closed}
The space $\mc{P}^\flat_{\bm{a}}\subset\mc{K}^d$ is closed w.r.t.\ $\dist_H$,
and $\interior\mc{P}^\flat_{\bm{a}}=\emptyset$.
\end{theorem}

\begin{proof}
Let $(\bm{x}_k)_k\subset\mc{C}^\flat_{\bm{a}}$ and $K\in\mc{K}^d$ with
$\lim_{k\to\infty}\dist_H(\co\{\bm{x}_k\},K)=0$.
Then $(\co\{\bm{x}_k\})_k$ is bounded, and since $\{\bm{x}_k\}\subset\co\{\bm{x}_k\}$ for all $k$,
it follows that $(\bm{x}_k)_k\subset\mc{C}^\flat_{\bm{a}}$ is bounded.
In particular, there exists an accumulation point $\bm{x}\in\R^{d\times N}$ of $(\bm{x}_k)_k$.
By Lemma \ref{lemma:polyhedral:convex:cone}, we have $\bm{x}\in\mc{C}^\flat_{\bm{a}}$.
By Corollary \ref{cor:embedding:Lipschitz}, the element $\co\{\bm{x}\}$ is an accumulation point
of $(\co\{\bm{x}_k\})_k$.
By uniqueness of the limit, we have $K=\co\{\bm{x}\}$.
All in all, the space $\mc{P}^\flat_{\bm{a}}\subset\mc{K}^d$ is closed.

\medskip

Let $P\in\mc{P}^\flat_{\bm{a}}$ and $\eps>0$.
The discussion following \cite[Theorem 3.4.1]{Schneider} provides
$K_\eps\in\mc{K}^d$ such that $\dist_H(P,K_\eps)\le\eps$ and $\partial K_\eps$
is a diffeomorphic image of the unit sphere, which implies that $K_\eps$ is not a polytope,
see e.g.\ \cite[Corollary 1.7.3 and Theorem 1.7.4]{Schneider}.
Hence $\interior\mc{P}^\flat_{\bm{a}}=\emptyset$.
\end{proof}

\begin{example}\label{ex:minkowski}
The space $\mc{P}^\flat_{\bm{a}}\subset\mc{K}^d$ is not convex w.r.t.\ Minkowski addition and scalar
multiplication:
Let $a_1=(1,0)^T$ and $a_2=(-1,0)^T$, and let elements $\bm{x},\bm{y}\in\interior\mc{C}^\flat_{\bm{a}}$ be given by
$x_1=(1,1)^T$, $x_2=(-1,-1)^T$, $y_1=(1,-1)^T$ and $y_2=(-1,1)^T$.
Then
\[\frac12\co\{\bm{x}\}+\frac12\co\{\bm{y}\}=\co\{(1,0)^T,(0,1)^T,(-1,0)^T,(0,-1)^T\}\notin\mc{P}^\flat_{\bm{a}}.\]
\end{example}

\begin{figure}[t]
\centering
\begin{tabular}{l@{\hspace{1.2cm}}c@{\hspace{1.2cm}}r}

\begin{tikzpicture}[scale=0.8]
\useasboundingbox (-1.8,-1.8) rectangle (1.8,1.8);

\coordinate (X1) at (1,1);
\coordinate (X2) at (-1,-1);
\coordinate (Y1) at (1,-1);
\coordinate (Y2) at (-1,1);

\draw[line width=1.5pt,blue!20] (X1)--(X2);
\draw[line width=1.5pt,blue!20] (Y1)--(Y2);

\fill (X1) circle (2pt) node[above right] {$x_1$};
\fill (X2) circle (2pt) node[below left] {$x_2$};
\fill (Y1) circle (2pt) node[below right] {$y_1$};
\fill (Y2) circle (2pt) node[above left] {$y_2$};

\draw[->,thick] (X1)--++(1,0) node[right] {$a_1$};
\draw[->,thick] (Y1)--++(1,0) node[right] {$a_1$};
\draw[->,thick] (X2)--++(-1,0) node[left] {$a_2$};
\draw[->,thick] (Y2)--++(-1,0) node[left] {$a_2$};

\node at (1.5,0.3) {$\co\{\bm{x}\}$};
\node at (-1.55,0.3) {$\co\{\bm{y}\}$};

\end{tikzpicture}

&

\begin{tikzpicture}[scale=0.8]
\useasboundingbox (-1.8,-1.8) rectangle (1.8,1.8);

\fill (0,0) circle (2pt);

\draw[->,thick] (0,0)--(1,0) node[right] {$a_1$};
\draw[->,thick] (0,0)--(-1,0) node[left] {$a_2$};

\end{tikzpicture}

&

\begin{tikzpicture}[scale=0.8]
\useasboundingbox (-1.8,-1.8) rectangle (1.8,1.8);

\coordinate (N) at (0,1);
\coordinate (E) at (1,0);
\coordinate (S) at (0,-1);
\coordinate (W) at (-1,0);

\fill[blue!20] (N)--(E)--(S)--(W)--cycle;
\draw[thick] (N)--(E)--(S)--(W)--cycle;

\foreach \P in {N,E,S,W}
  \fill (\P) circle (2pt);

\draw[->,thick] (E)--++(1,0) node[right] {$a_1$};
\draw[->,thick] (W)--++(-1,0) node[left] {$a_2$};

\node at (0,-1.45)
{$\frac12\co\{\bm{x}\}+\frac12\co\{\bm{y}\}$};
\end{tikzpicture}

\end{tabular}
\caption{Configuration discussed in Example \ref{ex:minkowski}.
We have $\co\{\bm{x}\}\in\mc{P}^\flat_{\bm{a}}$ and $\co\{\bm{y}\}\in\mc{P}^\flat_{\bm{a}}$
but $\frac12\co\{\bm{x}\}+\frac12\co\{\bm{y}\}\notin\mc{P}^\flat_{\bm{a}}$.}
\label{fig:minkowski}
\end{figure}
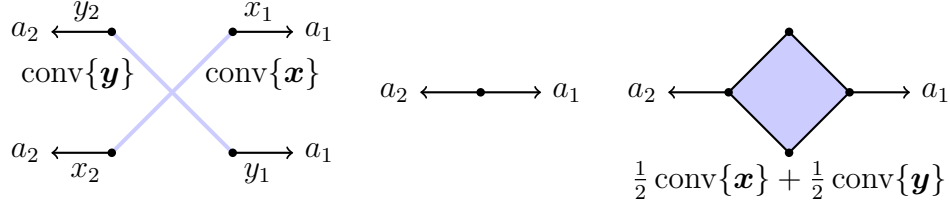

\section{The cone $\interior\mc{C}^\flat_{\bm{a}}$}
\label{sec:interior}

Now we characterize the interior of $\mc{C}^\flat_{\bm{a}}$.

\begin{theorem}\label{theorem:int:CA}
We have
$\interior\mc{C}^\flat_{\bm{a}}=\{\bm{x}\in\R^{d\times N}:a_i^Tx_j<a_i^Tx_i,\ i\neq j\}$
and for any $t>0$, we have $t\bm{a}\in\interior\mc{C}^\flat_{\bm{a}}$.
\end{theorem}

\begin{proof}
Continuity of $x\mapsto a_i^Tx$ for all $i$ and standard arguments imply the first assertion.
Let $i\neq j$.
Since $a_i\neq a_j$ and $\|a_k\|=1$ for all $k$,
we have $a_i^T(ta_j)<t=a_i^T(ta_i)$, which shows the second statement.
\end{proof}

The following result is a recipe for generating an interior point
near a given point $\bm{x}\in\mc{C}^\flat_{\bm{a}}$.
We have $\bm{z}=t\bm{a}$ in mind.

\begin{corollary}
For every $\bm{x}\in\mc{C}^\flat_{\bm{a}}$, $\bm{z}\in\interior\mc{C}^\flat_{\bm{a}}$
and $\lambda\in(0,1]$ we have
\[\bm{y}(\lambda):=(1-\lambda)\bm{x}+\lambda\bm{z}\in\interior\mc{C}^\flat_{\bm{a}}\]
and $\dist_{\mc{C}^\flat_{\bm{a}}}(\bm{x},\bm{y}(\lambda))
=\lambda\dist_{\mc{C}^\flat_{\bm{a}}}(\bm{x},\bm{z})$.
\end{corollary}

\begin{proof}
For $i\neq j$, Theorem \ref{theorem:int:CA} and
$\bm{z}\in\interior\mc{C}^\flat_{\bm{a}}$ imply that
\[
a_i^Ty_j(\lambda)
=a_i^T((1-\lambda)x_j+\lambda z_j)
<a_i^T((1-\lambda)x_i+\lambda z_i)
=a_i^Ty_i(\lambda),
\]
so again by Theorem \ref{theorem:int:CA}, we have
$\bm{y}(\lambda)\in\interior\mc{C}^\flat_{\bm{a}}$.
Finally, we have
\begin{align*}
\dist_{\mc{C}^\flat_{\bm{a}}}(\bm{x},\bm{y}(\lambda))
&=\max_i\|x_i-y_i(\lambda)\|
=\lambda\max_i\|x_i-z_i\|
=\lambda\dist_{\mc{C}^\flat_{\bm{a}}}(\bm{x},\bm{z}).
\end{align*}
\end{proof}

For $\bm{x}\in\interior\mc{C}^\flat_{\bm{a}}$ no $x_i$ is redundant.
This is a key property we were aiming for with our construction.

\begin{theorem}\label{theorem:interior:vertices}
Let $\bm{x}\in\interior\mc{C}^\flat_{\bm{a}}$.
Then we have
\[\{\bm{x}\}=\ext(\co\{\bm{x}\})
\quad\text{and}\quad
\#\{\bm{x}\}=\#\ext(\co\{\bm{x}\})=N.\]
\end{theorem}

\begin{proof}
By \cite[Proposition 2.2]{Ziegler} we have $\ext(\co\{\bm{x}\})\subset\{\bm{x}\}$.
Since Theorem \ref{theorem:int:CA} gives $a_i^Tx_j<a_i^Tx_i$ for $i\neq j$,
Lemma \ref{lem:strict:inequality:vertex} yields $x_i\in\ext(\co\{\bm{x}\})$ for all $i$.
All in all, we see that $\ext(\co\{\bm{x}\})=\{\bm{x}\}$.
Since by Lemma \ref{lem:strict:inequality:vertex}, we have $\#\{\bm{x}\}=N$,
the second statement ensues.
\end{proof}

The following result shows that any redundancy in the parameterization
$\mc{C}^\flat_{\bm{a}}$ occurs on the boundary.

\begin{corollary}\label{cor:co:injective}
The restriction $\co:\interior\mc{C}^\flat_{\bm{a}}\to\mc{P}^\flat_{\bm{a}}$
is injective.
\end{corollary}

\begin{proof}
Let $\bm{x},\bm{y}\in\interior\mc{C}^\flat_{\bm{a}}$ be such that
$\co\{\bm{x}\}=\co\{\bm{y}\}$.
By Theorem \ref{theorem:interior:vertices} we have
\[\{\bm{x}\}=\ext(\co\{\bm{x}\})=\ext(\co\{\bm{y}\})=\{\bm{y}\},\]
and by Theorem \ref{theorem:int:CA} and Lemma \ref{lem:strict:inequality:vertex},
we have $x_i\neq x_j$ and $y_i\neq y_j$ for $i\neq j$.
Assume that $x_i\neq y_i$ for some $i$.
Then $x_i=y_j$ for some $j\neq i$ and $y_i=x_k$ for some $k\neq i$,
and we obtain the contradiction
\[a_i^Ty_j=a_i^Tx_i>a_i^Tx_k=a_i^Ty_i>a_i^Ty_j.\]
\end{proof}

Theorem \ref{theorem:interior:vertices} does not extend to $\partial\mc{C}^\flat_{\bm{a}}$.

\begin{example}
Let $a_1=(1,0)^T$, $a_2=(0,1)^T$ and $a_3=(-1,0)^T$, and let
$\bm{x}=(x_1,x_2,x_3)\in\R^{2\times 3}$ be given by
$x_1=a_1$, $x_2=0$ and $x_3=a_3$.
It is easy to check that $\bm{x}\in\mc{C}^\flat_{\bm{a}}$,
but $\co\{\bm{x}\}=[x_1,x_3]$ and $\ext(\co\{\bm{x}\})=\{x_1,x_3\}$.
\end{example}

The following result complements Corollary \ref{cor:co:injective}.

\begin{proposition}\label{prop:interior:andLbdry:dont:mingle}
Let $\bm{x}\in\interior\mc{C}^\flat_{\bm{a}}$ and $\bm{y}\in\partial\mc{C}^\flat_{\bm{a}}$.
Then $\co\{\bm{x}\}\neq\co\{\bm{y}\}$.
\end{proposition}

\begin{proof}
Assume that $\co\{\bm{x}\}=\co\{\bm{y}\}$.
By Theorem \ref{theorem:interior:vertices} and \cite[Proposition 2.2]{Ziegler}, we have
\[\{\bm{x}\}=\ext(\co\{\bm{x}\})=\ext(\co\{\bm{y}\})\subset\{\bm{y}\}.\]
Theorem \ref{theorem:int:CA} and Lemma \ref{lem:strict:inequality:vertex} give
$\#\{\bm{x}\}=N$, and we have $\#\{\bm{y}\}\le N$,
so $\{\bm{x}\}=\{\bm{y}\}$.
Assume there exists $i$ with $x_i\neq y_i$.
Then there are $j$ with $x_i=y_j$ and $i\neq j$, and $k$ with $x_k=y_i$ and $i\neq k$, and
\[a_i^Ty_j=a_i^Tx_i>a_i^Tx_k=a_i^Ty_i\]
contradicts $\bm{y}\in\mc{C}^\flat_{\bm{a}}$.
Hence we have $\bm{x}=\bm{y}$, which contradicts
$\bm{x}\in\interior\mc{C}^\flat_{\bm{a}}\not\ni\bm{y}$.
All in all, the initial assumption is false, and $\co\{\bm{x}\}\neq\co\{\bm{y}\}$ holds.
\end{proof}

Corollary \ref{cor:co:injective} does not extend to the boundary.

\begin{example}
Let $a_1=(1,0)^T$, $a_2=(0,1)^T$ and $a_3=(-1,0)^T$, and let
$\bm{x}(s)=(x_1(s),x_2(s),x_3(s))\in\R^{2\times 3}$ be given by
$x_1(s)=a_1$, $x_2(s)=(s,0)^T$ and $x_3(s)=a_3$ for $s\in[-1,1]$.
It is easy to check that $\bm{x}(s)\in\mc{C}^\flat_{\bm{a}}$ for $s\in[-1,1]$ and
$\bm{x}(s)\neq\bm{x}(t)$ for $s\neq t$,
but $\co\{\bm{x}(s)\}=[x_1,x_3]$ for all $s\in[-1,1]$.
\end{example}

The following theorem is the analog of Lemma \ref{lem:normal:cone:relationship} for
$\interior\mc{C}^\flat_{\bm{a}}$.
Points in $\interior\mc{C}^\flat_{\bm{a}}$ encode polytopes with full-dimensional
normal cones at every vertex.

\begin{theorem}\label{temporary}
Let $\bm{x}\in\R^{d\times N}$.
Then $\bm{x}\in\interior\mc{C}^\flat_{\bm{a}}$ if and only if $x_i\neq x_j$ for $i\neq j$ and
$a_i\in\interior\mc{N}_{\co\{\bm{x}\}}(x_i)$ for all $i$.
\end{theorem}

\begin{proof}
Let $\bm{x}\in\interior\mc{C}^\flat_{\bm{a}}$ and fix $i$.
By Theorem \ref{theorem:int:CA} and Lemma \ref{lem:strict:inequality:vertex}
we have $a_i^Tx_j<a_i^Tx_i$ and $x_i\neq x_j$ for $i\neq j$.
Then $\eps:=\min_{\{j:i\neq j\}}\frac{a_i^T(x_i-x_j)}{\|x_i-x_j\|}>0$.
Let $a:=a_i+z$ with $z\in B_\eps(0)$, and let $x\in\co\{\bm{x}\}$.
Then there exists $\lambda\in\R^N_+$ with $\mathbbm{1}^T\lambda=1$
and $\sum_{j=1}^N\lambda_jx_j=x$, and we compute
\begin{align*}
a^T(x-x_i)
&=\sum_{j\neq i}\lambda_j(a_i+z)^T(x_j-x_i)\\
&=\sum_{j\neq i}\lambda_j\tfrac{a_i^T(x_j-x_i)}{\|x_j-x_i\|}\|x_j-x_i\|
+\sum_{j\neq i}\lambda_jz^T(x_j-x_i)\\
&\le-\eps\sum_{j=1}^N\lambda_j\|x_i-x_j\|+\|z\|\sum_{j=1}^N\lambda_j\|x_i-x_j\|\le0.
\end{align*}
We conclude that $a\in\mc{N}_{\co\{\bm{x}\}}(x_i)$ for all $a\in B_\eps(a_i)$ and, consequently, that
$a_i\in\interior\mc{N}_{\co\{\bm{x}\}}(x_i)$.

\medskip

Conversely, let $\bm{x}\in\R^{d\times N}$ with $x_i\neq x_j$ for $i\neq j$
and $a_i\in\interior\mc{N}_{\co\{\bm{x}\}}(x_i)$ for all $i$.
Assume there exist 
$i\neq j$ with $a_i^Tx_i\le a_i^Tx_j$.
For any $\eps>0$ let $a(\eps):=a_i+\eps(x_j-x_i)$.
Then $\lim_{\eps\to0}a(\eps)=a_i$ and
\[
a(\eps)^T(x_j-x_i)
=(a_i+\eps(x_j-x_i))^T(x_j-x_i)
\ge\eps\|x_j-x_i\|^2>0
\]
shows that $a_i\notin\interior\mc{N}_{\co\{\bm{x}\}}(x_i)$.
Hence the assumption was wrong, we have $a_i^Tx_j<a_i^Tx_i$ for $i\neq j$,
and in view of Theorem \ref{theorem:int:CA}, we have $\bm{x}\in\interior\mc{C}^\flat_{\bm{a}}$.
\end{proof}

It is tempting to think that for any $\bm{x}\in\interior\mc{C}^\flat_{\bm{a}}$
the corresponding $\co\{\bm{x}\}$ should have full dimension.
This is only true in dimension $d=2$.

\begin{theorem}
The following statements hold.
\begin{itemize}
\item [a)] When $d=2$ and $\#\{\bm{a}\}\ge3$, then $\dim(\co\{\bm{x}\})=2$
for any $\bm{x}\in\interior\mc{C}^\flat_{\bm{a}}$.
\item [b)] For every $d>2$ there exist $\bm{a}\in\R^{d\times N}$ satisfying
the standing assumptions and $\bm{x}\in\interior\mc{C}^\flat_{\bm{a}}$ with
$\dim(\co\{\bm{x}\})\le2$.
\end{itemize}
\end{theorem}

\begin{proof}
We use Theorem \ref{theorem:int:CA} throughout the proof.

\medskip

a) By Theorem \ref{theorem:interior:vertices} we have $\#\ext(\co\{\bm{x}\})\ge3$,
and three distinct extreme points cannot be collinear.

\medskip

b) Let $d>2$, let $N>0$, let $(\phi_j)_{j=1}^N\subset(0,\pi)^{d-2}\times(0,2\pi)$ with
\[0<[\phi_1]_{d-1}<[\phi_2]_{d-1}<\ldots[\phi_N]_{d-1}<2\pi,\]
and for $i\in\{1,\ldots,N\}$ consider the vectors
\begin{align*}
a_i&:=(\cos[\phi_i]_1,\sin[\phi_i]_1\cos[\phi_i]_2,\ldots,
(\prod_{j=1}^{d-2}\sin[\phi_i]_j)\cos[\phi_i]_{d-1},\prod_{j=1}^{d-1}\sin[\phi_i]_j),\\
x_i&:=(0,\ldots,0,\cos[\phi_i]_{d-1},\sin[\phi_i]_{d-1}).
\end{align*}
Then $\bm{a}:=(a_1,\ldots,a_N)$ satisfies the standing assumptions, and since for $i\neq j$
the strict Cauchy-Schwarz inequality gives
\begin{align*}
a_i^Tx_j
&=\big(\cos[\phi_i]_{d-1}\cos[\phi_j]_{d-1}+\sin[\phi_i]_{d-1}\sin[\phi_j]_{d-1}\big)
\prod_{j=1}^{d-2}\sin[\phi_i]_j
<\prod_{j=1}^{d-2}\sin[\phi_i]_j\\
&=\big(\cos[\phi_i]_{d-1}\cos[\phi_i]_{d-1}+\sin[\phi_i]_{d-1}\sin[\phi_i]_{d-1}\big)
\prod_{j=1}^{d-2}\sin[\phi_i]_j
=a_i^Tx_i,
\end{align*}
we have $\bm{x}:=(x_1,\ldots,x_N)\in\interior\mc{C}^\flat_{\bm{a}}$
with $\co\{\bm{x}\}\subset\{0\}\times\ldots\times\{0\}\times\R^2$.
\end{proof}

While part b) looks like a negative result, it is indeed very positive:
It means that we can underapproximate at least some lower-dimensional convex bodies
with $\co\{\bm{x}\}$ where $\bm{x}\in\interior\mc{C}^\flat_{\bm{a}}$.

\section{Redundancy in the constraints}
\label{sec:redundancy}

In the following, when we state that a constraint is redundant,
we mean that it is individually redundant with respect to all other constraints.
We do not mean to imply that all individually redundant constraints
are simultaneously redundant.

\medskip

The following lemma states that whenever $a_k\in\cone(a_i,a_j)\setminus\{0\}$,
then two of the constraints defining $\mc{C}_{\bm{a}}^\flat$ are redundant.

\begin{lemma}
\label{lem:redundancy:positive:combination}
Let $i,j,k\in\{1,\ldots,N\}$ be pairwise distinct and assume that
\begin{equation}\label{positive:combination}
a_k=\alpha a_i+\beta a_j
\end{equation}
for some $\alpha,\beta>0$.
Let $\bm{x}\in\R^{d\times N}$ satisfy
\begin{equation}\label{satisfied:inequalities}
a_i^Tx_k\le a_i^Tx_i,\quad a_j^Tx_k\le a_j^Tx_j,\quad
a_k^Tx_i\le a_k^Tx_k\quad\text{and}\quad a_k^Tx_j\le a_k^Tx_k.
\end{equation}
Then we also have
\[a_i^Tx_j\le a_i^Tx_i\quad\text{and}\quad a_j^Tx_i\le a_j^Tx_j.\]
\end{lemma}

\begin{proof}
Equation \eqref{positive:combination} yields
\[a_k^Tx_j-a_k^Tx_k=\alpha(a_i^Tx_j-a_i^Tx_k)+\beta(a_j^Tx_j-a_j^Tx_k)\]
and hence
\[a_i^Tx_j-a_i^Tx_k=\frac1\alpha(a_k^Tx_j-a_k^Tx_k)+\frac\beta\alpha(a_j^Tx_k-a_j^Tx_j).\]
Adding $a_i^Tx_k-a_i^Tx_i$ and using \eqref{satisfied:inequalities} gives
\[a_i^Tx_j-a_i^Tx_i
=(a_i^Tx_k-a_i^Tx_i)+\frac\beta\alpha(a_j^Tx_k-a_j^Tx_j)+\frac1\alpha(a_k^Tx_j-a_k^Tx_k)
\le0,\]
which proves the first inequality.
The second inequality is proved in the same way.
\end{proof}

Constraints relating antipodal directions and their corresponding points
are redundant as well.

\begin{lemma}\label{lem:redundancy:antipodal}
Let $i,j,k\in\{1,\ldots,N\}$ be pairwise distinct and suppose that $a_j=-a_i$.
Let $\bm x\in\R^{d\times N}$ satisfy
\[a_i^Tx_k\le a_i^Tx_i\quad\text{and}\quad a_j^Tx_k\le a_j^Tx_j .\]
Then we also have
\[a_i^Tx_j\le a_i^Tx_i\quad\text{and}\quad a_j^Tx_i\le a_j^Tx_j .\]
\end{lemma}

\begin{proof}
Using the assumptions we directly estimate
\[a_i^Tx_j=-a_j^Tx_j\le-a_j^Tx_k=a_i^Tx_k\le a_i^Tx_i.\]
The second statement can be shown in the same way.
\end{proof}

Both of the mechanisms described above occur frequently in the standard discretization
given by mapping a regular grid to the sphere by means of polar coordinates.

\section{Results on projectors}
\label{sec:projectors}

We examine the projectors $\pi^\flat_{\bm{a}}$ and $\Pi^\flat_{\bm{a}}$
from \eqref{def:pi} and \eqref{def:Pi}.

\begin{theorem}\label{thm:basic:properties:projectors}
The mappings $\pi^\flat_{\bm{a}}:\mc{K}^d\rightrightarrows\mc{C}^\flat_{\bm{a}}$
and $\Pi^\flat_{\bm{a}}:\mc{K}^d\rightrightarrows\mc{P}^\flat_{\bm{a}}$ are well-defined.
Furthermore 
$\pi^\flat_{\bm{a}}:\mc{K}^d\to\mc{K}(\R^{d\times N})$
and $\Pi^\flat_{\bm{a}}:\mc{K}^d\to\mc{C}(\mc{K}^d)$,
where $\mc{K}(\R^{d\times N})$ denotes the space of all nonempty convex and compact subsets
of $\R^{d\times N}$ and $\mc{C}(\mc{K}^d)$ is the space of all
nonempty and compact subsets of $\mc{K}^d$.
\end{theorem}

\begin{proof}
Let $K\in\mc{K}^d$.
By construction we have $\pi^\flat_{\bm{a}}(K)\subset\mc{C}^\flat_{\bm{a}}$.
Since $\pi^\flat_{\bm{a}}(K)$ is the Cartesian product of the nonempty, convex
and compact faces $\argmax_{x\in K}a_i^Tx$ of $K$, the set $\pi^\flat_{\bm{a}}(K)$
has those same properties.
It follows that $\Pi^\flat_{\bm{a}}(K)\neq\emptyset$,
and by Corollary \ref{cor:embedding:Lipschitz} the set $\Pi^\flat_{\bm{a}}(K)$
is the continuous image of the compact set $\pi^\flat_{\bm{a}}(K)$ and hence compact itself.
\end{proof}

The projector $\Pi^\flat_{\bm{a}}$ maps sets back into themselves.

\begin{lemma}\label{lemma:subset}
For every $K\in\mc{K}^d$ and all $P\in\Pi^\flat_{\bm{a}}(K)$ we have $P\subset K$.
\end{lemma}

\begin{proof}
For any $\bm{x}\in\pi_{\bm{a}}^\flat(K)$, by definition $x_i\in K$ for all $i$.
Since $K\in\mc{K}^d$, we have $\co\{\bm{x}\}\subset K$.
\end{proof}

\begin{theorem}\label{construction:site}
The mappings
$\pi^\flat_{\bm{a}}:\mc{K}^d\rightrightarrows\mc{C}^\flat_{\bm{a}}$
and
$\Pi^\flat_{\bm{a}}:\mc{K}^d\rightrightarrows\mc{P}^\flat_{\bm{a}}$
are surjective, and the following statements hold:
\begin{itemize}
\item [a)] We have $\bm{x}\in\pi^\flat_{\bm{a}}(\co\{\bm{x}\})$ for all $\bm{x}\in\mc{C}^\flat_{\bm{a}}$.
\item [b)] We have $\bm{x}=\pi^\flat_{\bm{a}}(\co\{\bm{x}\})$ for all $\bm{x}\in\interior\mc{C}^\flat_{\bm{a}}$.
\item [c)] We have $P\in\Pi^\flat_{\bm{a}}(P)$ for all $P\in\mc{P}^\flat_{\bm{a}}$.
\item [d)] We have $P=\Pi^\flat_{\bm{a}}(P)$ for all $P\in\co(\interior\mc{C}^\flat_{\bm{a}})$.
\end{itemize}
 \end{theorem}

\begin{proof}
Let $\bm{x}\in\mc{C}^\flat_{\bm{a}}$. Then $\co\{\bm{x}\}\in\mc{K}^d$ and
$a_i^Tx_j\le a_i^Tx_i$ for all $i$ and $j$.
Let $x\in\co\{\bm{x}\}$.
Then there exists $\lambda\in\R^N_+$
with $\mathbbm{1}^T\lambda=1$ and $\sum_{j=1}^N\lambda_jx_j=x$.
For every $i$ it follows that
\begin{equation*}
a_i^Tx=\sum_{j=1}^N\lambda_ja_i^Tx_j\le\sum_{j=1}^N\lambda_ja_i^Tx_i=a_i^Tx_i.
\end{equation*}
Since $x\in\co\{\bm{x}\}$ was arbitrary, this implies $x_i\in\argmax_{x\in\co\{\bm{x}\}}a_i^Tx$.
All in all, we have shown a), and it follows that
$\pi^\flat_{\bm{a}}:\mc{K}^d\rightrightarrows\mc{C}^\flat_{\bm{a}}$ is surjective.

\medskip

Let $\bm{x}\in\interior\mc{C}^\flat_{\bm{a}}$, fix $i$, and let $x\in\co\{\bm{x}\}$
with $x\neq x_i$.
Then there exist $\lambda\in\R^N_+$ and $k\neq i$
with $\mathbbm{1}^T\lambda=1$, $\lambda_k\neq 0$ and $\sum_{j=1}^N\lambda_jx_j=x$.
By Theorem \ref{theorem:int:CA} we have $a_i^Tx_j<a_i^Tx_i$, and it follows that
\begin{equation*}
a_i^Tx=\sum_{j=1}^N\lambda_ja_i^Tx_j<\sum_{j=1}^N\lambda_ja_i^Tx_i=a_i^Tx_i.
\end{equation*}
Hence $x_i=\argmax_{x\in\co\{\bm{x}\}}a_i^Tx$.
Since $i$ was arbitrary, statement b) holds.

\medskip

Let $P\in\mc{P}^\flat_{\bm{a}}$.
Then there exists $\bm{x}\in\mc{C}^\flat_{\bm{a}}$ with $P=\co\{\bm{x}\}$.
By a) we have  $\bm{x}\in\pi^\flat_{\bm{a}}(\co\{\bm{x}\})$,
and applying $\co$ to both sides yields
\[P=\co\{\bm{x}\}\in\Pi^\flat_{\bm{a}}(\co\{\bm{x}\})=\Pi^\flat_{\bm{a}}(P),\]
which is c).
In particular, the map $\Pi^\flat_{\bm{a}}:\mc{K}^d\rightrightarrows\mc{P}^\flat_{\bm{a}}$ is surjective.
Statement d) follows in the same way from b).
\end{proof}

The map $\Pi^\flat_{\bm{a}}$ is not a minimal distance projector w.r.t.\ $\dist_H$.

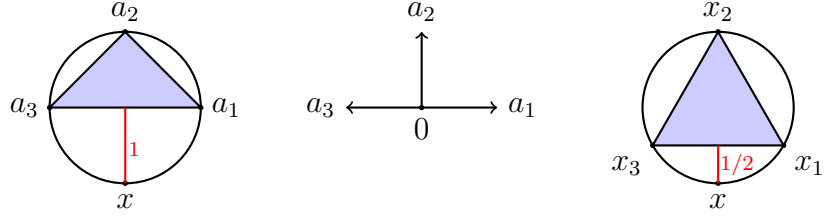
\begin{figure}
\centering
\begin{tabular}{lcr}

\begin{tikzpicture}[scale=1]
  \useasboundingbox (-1.75,-1.3) rectangle (1.75,1.3);
  \draw[thick] (0,0) circle (1);

  \coordinate (a1) at (1,0);
  \coordinate (a2) at (0,1);
  \coordinate (a3) at (-1,0);

  \fill (a1) circle (1pt);
  \node [right] at (a1) {$a_1$};
  \fill (a2) circle (1pt);
  \node [above] at (a2) {$a_2$};
  \fill (a3) circle (1pt);
  \node [left] at (a3) {$a_3$};

  \fill[blue!20] (a1) -- (a2) -- (a3) -- cycle;
  \draw[thick] (a1) -- (a2) -- (a3) -- cycle;

  \coordinate (south) at (0,-1);
  \fill (south) circle (1pt);
  \node[below] at (south) {$x$};

  \draw[red, thick] (south) -- (0,0);
  \node[red] at (0.12,-0.55) {$\scriptstyle 1$};
\end{tikzpicture}

&

\begin{tikzpicture}[scale=1]
  \useasboundingbox (-1.75,-1.3) rectangle (1.75,1.3);
  \fill (0,0) circle (1pt);
  \node[below] at (0,0) {$0$};

  \coordinate (a1) at (1,0);
  \coordinate (a2) at (0,1);
  \coordinate (a3) at (-1,0);

  \draw[thick,->] (0,0) -- (a1);
  \draw[thick,->] (0,0) -- (a2);
  \draw[thick,->] (0,0) -- (a3);

  \node[right] at (a1) {$a_1$};
  \node[above] at (a2) {$a_2$};
  \node[left] at (a3) {$a_3$};
\end{tikzpicture}

&

\begin{tikzpicture}[scale=1]
  \useasboundingbox (-1.75,-1.3) rectangle (1.75,1.3);
  \draw[thick] (0,0) circle (1);

  \coordinate (x1) at ({sqrt(3)/2}, {-1/2});
  \coordinate (x2) at (0,1);
  \coordinate (x3) at ({-sqrt(3)/2}, {-1/2});

  \fill (x1) circle (1pt);
  \node [below right] at (x1) {$x_1$};
  \fill (x2) circle (1pt);
  \node [above] at (x2) {$x_2$};
  \fill (x3) circle (1pt);
  \node [below left] at (x3) {$x_3$};

  \fill[blue!20] (x1) -- (x2) -- (x3) -- cycle;
  \draw[thick] (x1) -- (x2) -- (x3) -- cycle;

  \coordinate (south) at (0,-1);
  \fill (south) circle (1pt);
  \node[below] at (south) {$x$};

  \draw[red, thick] (south) -- (0,-0.5);
  \node[red] at (0.25,-0.75) {$\scriptstyle 1/2$};
\end{tikzpicture}

\end{tabular}
\caption{For $K=B_1(0)$ the projection $\Pi^\flat_{\bm{a}}(K)$ (left)
induced by $\bm{a}$ (middle) is not the 
closest set (right)
in $\mc{P}^\flat_{\bm{a}}$ to $K$.
This illustrates Example \ref{counterexample:hausdorff:projector}.}
\end{figure}

\begin{example}\label{counterexample:hausdorff:projector}
Let $d=2$ and consider
\[a_1=(1,0)^T,\quad a_2=(0,1)^T,\quad a_3=(-1,0)^T.\]
For $K:=B_1(0)$ we have $\pi^\flat_{\bm{a}}(K)=(a_1,a_2,a_3)$, and with $x:=(0,-1)^T$ we find
\[\dist_H(K,\Pi^\flat_{\bm{a}}(K))=\dist(x,\Pi^\flat_{\bm{a}}(K))=1.\]
But it is easy to check that the points
\[x_1=(\tfrac{\sqrt{3}}{2},-\tfrac12)^T,\
x_2=(0,1)^T,\
x_3=(-\tfrac{\sqrt{3}}{2},-\tfrac12)^T\]
satisfy $(x_1,x_2,x_3)\in\mc{C}^\flat_{\bm{a}}$.
Hence $P:=\co\{x_1,x_2,x_3\}\in\mc{P}^\flat_{\bm{a}}$ with
\[\dist_H(K,P)=\dist(x,P)=\tfrac12,\]
and $\Pi^\flat_{\bm{a}}(K)\notin\argmin_{P'\in\mc{P}^\flat_{\bm{a}}}\dist_H(K,P')$.
\end{example}

\begin{figure}
\centering
\def\epsfig{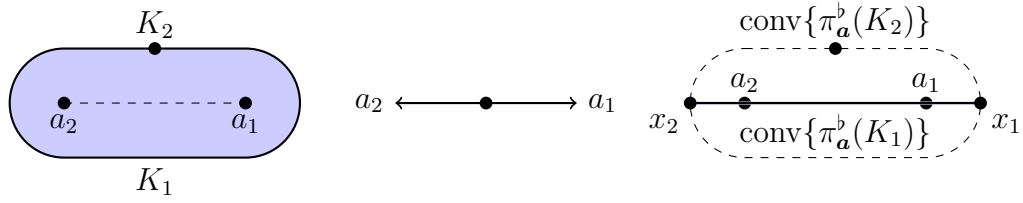}

\begin{tabular}{ccc}

\begin{tikzpicture}[scale=1.2]
\useasboundingbox (-1.8,-.8) rectangle (1.8,1.1);

\coordinate (A1) at (1,0);
\coordinate (A2) at (-1,0);
\coordinate (K1) at (0,-\epsfig);
\coordinate (K2) at (0,\epsfig);

\fill[blue!20]
  (-1,\epsfig) -- (1,\epsfig)
  arc (90:-90:\epsfig)
  -- (-1,-\epsfig)
  arc (-90:-270:\epsfig) -- cycle;
\draw[thick]
  (-1,\epsfig) -- (1,\epsfig)
  arc (90:-90:\epsfig)
  -- (-1,-\epsfig)
  arc (-90:-270:\epsfig) -- cycle;

\draw[dashed] (A2)--(A1);
\fill (A1) circle (2pt) node[below] {$a_1$};
\fill (A2) circle (2pt) node[below] {$a_2$};
\fill (K2) circle (2pt) node[above] {$K_2$};
\node[below] at (K1) {$K_1$};

\end{tikzpicture}

&

\begin{tikzpicture}[scale=1.2]
\useasboundingbox (-1.5,-.8) rectangle (1.7,1.1);

\coordinate (O) at (0,0);

\fill (O) circle (2pt);
\draw[->,thick] (O)--(1,0) node[right] {$a_1$};
\draw[->,thick] (O)--(-1,0) node[left] {$a_2$};

\end{tikzpicture}

&

\begin{tikzpicture}[scale=1.2]
\useasboundingbox (-1.8,-.8) rectangle (1.8,1.1);

\coordinate (A1) at (1,0);
\coordinate (A2) at (-1,0);
\coordinate (X1) at ({1+\epsfig},0);
\coordinate (X2) at ({-1-\epsfig},0);
\coordinate (Y) at (0,\epsfig);
\coordinate (K1) at (0,-\epsfig);

\draw[dashed]
  (-1,\epsfig) -- (1,\epsfig)
  arc (90:-90:\epsfig)
  -- (-1,-\epsfig)
  arc (-90:-270:\epsfig) -- cycle;

\fill (A1) circle (2pt);
\fill (A2) circle (2pt);
\node[above] at (A1) {$a_1$};
\node[above] at (A2) {$a_2$};

\draw[very thick,blue!20] (X2)--(X1);
\draw[thick] (X2)--(X1);
\fill (X1) circle (2pt) node[below right] {$x_1$};
\fill (X2) circle (2pt) node[below left] {$x_2$};

\fill (Y) circle (2pt);
\node[above] at (Y) {$\co\{\pi^\flat_{\bm a}(K_2)\}$};
\node[below] at (0,0) {$\co\{\pi^\flat_{\bm a}(K_1)\}$};

\end{tikzpicture}
\end{tabular}
\caption{Situation from Example \ref{ex:not:nested}.
We have $K_2\subset K_1$ (shaded region in left sketch),
but $\co\{\pi^\flat_{\bm{a}}(K_2)\}$ (isolated point in right sketch)
is not a subset of $\co\{\pi^\flat_{\bm{a}}(K_1)\}$ (solid black line in right sketch).}
\end{figure}

The projector $\Pi^\flat_{\bm{a}}$ is not monotone w.r.t.\ set inclusion.

\begin{example}\label{ex:not:nested}
Let $a_1=(1,0)^T$ and $a_2=(-1,0)^T$, and consider
\[K_1=\co\{a_1,a_2\}+B_{\eps}(0)\in\mc{K}^d,\quad K_2=\{(0,\eps)^T\}\in\mc{K}^d.\]
Then $K_2\subset K_1$, but we have
\[\pi^\flat_{\bm{a}}(K_1)=(x_1,x_2)=((1+\eps,0)^T,(-1-\eps,0)^T),\quad
\pi^\flat_{\bm{a}}(K_2)=((0,\eps)^T,(0,\eps)^T).\]
Obviously, we have $\co\{\pi^\flat_{\bm{a}}(K_2)\}\not\subset\co\{\pi^\flat_{\bm{a}}(K_1)\}$.
\end{example}

However, the projector $\pi^\flat_{\bm{a}}$ is compatible with Minkowski addition and
positive scaling.

\begin{lemma}\label{lem:Minkowski:compatible}
For all $K,L\in\mc{K}^d$ and $\lambda\ge0$, we have
\[\pi^\flat_{\bm{a}}(K+L)=\pi^\flat_{\bm{a}}(K)+\pi^\flat_{\bm{a}}(L)
\quad\text{and}\quad
\pi^\flat_{\bm{a}}(\lambda K)=\lambda\pi^\flat_{\bm{a}}(K),\]
where addition and scalar multiplication on the right-hand sides are understood componentwise in $\R^{d\times N}$.
\end{lemma}

\begin{proof}
Let $\bm{x}\in\pi^\flat_{\bm{a}}(K)$ and $\bm{y}\in\pi^\flat_{\bm{a}}(L)$,
and fix $i$.
Then $x_i\in K$ and $y_i\in L$,
and $a_i^Tx_i\ge a_i^Tx$ for all $x\in K$ and $a_i^Ty_i\ge a_i^Ty$ for all $y\in L$.
Since both $x_i+y_i\in K+L$ and $a_i^T(x_i+y_i)\ge a_i^T(x+y)$
for all $(x,y)\in K\times L$, we have $x_i+y_i\in\argmax_{z\in K+L}a_i^Tz$.
All in all, we find
$\pi^\flat_{\bm{a}}(K)+\pi^\flat_{\bm{a}}(L)\subset\pi^\flat_{\bm{a}}(K+L)$.

Conversely, let $\bm{z}\in\pi^\flat_{\bm{a}}(K+L)$ and fix $i$.
Since $z_i\in K+L$ there exist $x_i\in K$ and $y_i\in L$ with $z_i=x_i+y_i$, and
\[a_i^Tx_i+a_i^Ty_i=a_i^Tz_i\ge a_i^Tx+a_i^Ty\quad\forall\,(x,y)\in K\times L.\]
Picking $x=x_i$ forces $a_i^Ty_i\ge a_i^Ty$ for all $y\in L$,
and picking $y=y_i$ forces $a_i^Tx_i\ge a_i^Tx$ for all $x\in K$,
and hence $x_i\in\argmax_{x\in K}a_i^Tx$ and $y_i\in\argmax_{y\in L}a_i^Ty$.
All in all, we find
$\pi^\flat_{\bm{a}}(K+L)\subset\pi^\flat_{\bm{a}}(K)+\pi^\flat_{\bm{a}}(L)$.

The proof of the second statement is similar.
\end{proof}

Lemma \ref{lem:Minkowski:compatible} does not extend to $\Pi^\flat_{\bm{a}}$.

\begin{example}\label{ex:Pi:Minkowski}
Let $a_1=(1,0)^T$ and $a_2=(-1,0)^T$,
define $x_1=(1,1)^T$, $x_2=(-1,-1)^T$, $y_1=(1,-1)^T$ and $y_2=(-1,1)^T$,
and consider the line segments
$K:=\co\{x_1,x_2\}$ and $L:=\co\{y_1,y_2\}$.
Then $\pi^\flat_{\bm{a}}(K)=(x_1,x_2)$
and $\pi^\flat_{\bm{a}}(L)=(y_1,y_2)$,
and consequently,
$\Pi^\flat_{\bm{a}}(K)=K$,
$\Pi^\flat_{\bm{a}}(L)=L$ and
\[\Pi^\flat_{\bm{a}}(K)+\Pi^\flat_{\bm{a}}(L)
=K+L=\co\{( 2, 0)^T,( 0, 2)^T,(-2, 0)^T,( 0,-2)^T\},\]
which is the shaded region in Figure \ref{fig:Pi:Minkowski}.
On the other hand, we have $\pi^\flat_{\bm{a}}(K+L)=(( 2, 0)^T,(-2, 0)^T)$
and hence
\[\Pi^\flat_{\bm{a}}(K+L)=\co\{( 2, 0)^T,(-2, 0)^T\},\]
which is the horizontal line in Figure \ref{fig:Pi:Minkowski}.
\end{example}

\begin{figure}[t]
\centering
\begin{tabular}{ccc}

\begin{tikzpicture}[scale=0.7]
  \useasboundingbox (-1.7,-1.8) rectangle (1.9,1.8);

  \coordinate (X1) at (1,1);
  \coordinate (X2) at (-1,-1);

  \fill (X1) circle (2pt) node[above left] {$x_1$};
  \fill (X2) circle (2pt) node[below right] {$x_2$};

  \draw[very thick] (X1)--(X2);

  \draw[->,thick] (X1)--++(0.8,0) node[below] {$a_1$};
  \draw[->,thick] (X2)--++(-0.8,0) node[above] {$a_2$};

  \node at (0.5,-0.1) {$K$};
\end{tikzpicture}

&\hspace{0.3cm}

\begin{tikzpicture}[scale=0.7]
  \useasboundingbox (-1.9,-1.8) rectangle (1.9,1.8);

  \coordinate (Y1) at (1,-1);
  \coordinate (Y2) at (-1,1);

  \fill (Y1) circle (2pt) node[below left] {$y_1$};
  \fill (Y2) circle (2pt) node[above right] {$y_2$};

  \draw[very thick] (Y1)--(Y2);

  \draw[->,thick] (Y1)--++(0.8,0) node[above] {$a_1$};
  \draw[->,thick] (Y2)--++(-0.8,0) node[below] {$a_2$};

  \node at (-0.6,-0.1) {$L$};
\end{tikzpicture}

&\hspace{0.3cm}

\begin{tikzpicture}[scale=0.7]
  \useasboundingbox (-2.5,-1.8) rectangle (2.5,1.8);

  \coordinate (N) at (0,2);
  \coordinate (E) at (2,0);
  \coordinate (S) at (0,-2);
  \coordinate (W) at (-2,0);

  \fill [blue!20] (N)--(E)--(S)--(W)--cycle;
  \draw[thick] (N)--(E)--(S)--(W)--cycle;

  \draw[very thick] (E)--(W);

  \fill (N) circle (2pt);
  \fill (E) circle (2pt);
  \fill (S) circle (2pt);
  \fill (W) circle (2pt);

  \draw[->,thick] (E)--++(0.8,0) node[below] {$a_1$};
  \draw[->,thick] (W)--++(-0.8,0) node[above] {$a_2$};

  \node[right] at (1.2,-1) {$K+L$};
  \node[above] at (0,0) {\small$\Pi^\flat_{\bm{a}}(K+L)$};
  \node[above right] at ([xshift=-5pt]E) {$(2,0)^T$};
  \node[above right] at ([xshift=-5pt,yshift=0pt]N) {$(0,2)^T$};
\end{tikzpicture}

\end{tabular}
\caption{Situation from Example \ref{ex:Pi:Minkowski}.
The horizontal black line $\Pi^\flat_{\bm{a}}(K+L)$ does not coincide
with the shaded diamond $\Pi^\flat_{\bm{a}}(K)+\Pi^\flat_{\bm{a}}(L)=K+L$
in the right sketch.}
\label{fig:Pi:Minkowski}
\end{figure}
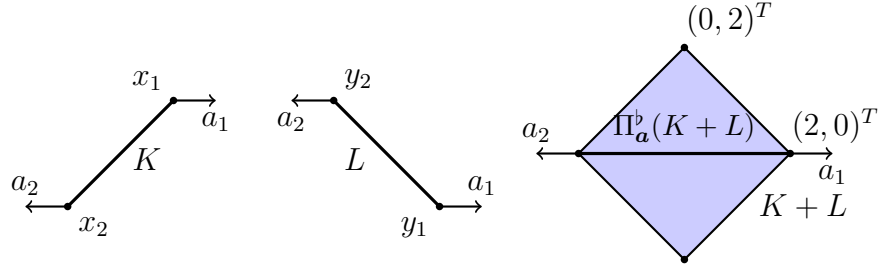

Now we turn to continuity properties of the projectors.
The proof of the following theorem can in principle be deduced
from the Berge maximum principle \cite[Theorem 16.31]{Aliprantis},
but a direct proof is simpler.

\begin{theorem}
The projectors $\pi^\flat_{\bm{a}}:\mc{K}^d\rightrightarrows\mc{C}^\flat_{\bm{a}}$
and $\Pi^\flat_{\bm{a}}:\mc{K}^d\rightrightarrows\mc{P}^\flat_{\bm{a}}$ are upper
semicontinuous, i.e.\ for every $K\in\mc{K}^d$ and $\eps>0$ there exists $\delta>0$
such that every $L\in\mc{K}^d$ with $\dist_H(K,L)\le\delta$ satisfies
\[\dist(\pi^\flat_{\bm{a}}(L),\pi^\flat_{\bm{a}}(K))\le\eps
\quad\text{and}\quad
\sup_{L'\in\Pi^\flat_{\bm{a}}(L)}\inf_{K'\in\Pi^\flat_{\bm{a}}(K)}
\dist_H(L',K')\le\eps.\]
\end{theorem}

\begin{proof}
If the first assertion is false, then there exist $K\in\mc{K}^d$, $\eps>0$ and sequences
$(K_k)_k\subset\mc{K}^d$ and $(\bm{x}_k)_k\subset\R^{d\times N}$
with $\bm{x}_k\in\pi^\flat_{\bm{a}}(K_k)$ for all $k$ and
\begin{align}
&\lim_{k\to\infty}\dist_H(K_k,K)=0,\label{local:dh}\\
&\dist(\bm{x}_k,\pi^\flat_{\bm{a}}(K))>\eps\quad\forall\,k\in\N.\label{local:separation}
\end{align}
By \eqref{local:dh} the sequence $(K_k)_k$ is bounded, and hence so is $(\bm{x}_k)_k$.
Thus there exist $\N'\subset\N$ and $\bm{x}\in\R^{d\times N}$
such that $\bm{x}=\lim_{\N'\ni k\to\infty}\bm{x}_k$.
Because of $\{\bm{x}_k\}\subset K_k$ for all $k$ and \eqref{local:dh}
we have $\{\bm{x}\}\subset K$.
Let $z\in K$. 
By \eqref{local:dh} there exist $z_k\in K_k$ for $k\in\N$ with $z=\lim_{k\to\infty}z_k$.
Since $\bm{x}_k\in\pi^\flat_{\bm{a}}(K_k)$, we have
\[a_i^Tz
=\lim_{\N'\ni k\to\infty}a_i^Tz_k
\le\lim_{\N'\ni k\to\infty}a_i^T[\bm{x}_k]_i
=a_i^T[\bm{x}]_i\]
for all $i$.
Since $z$ was arbitrary, this shows $\bm{x}\in\pi^\flat_{\bm{a}}(K)$,
which contradicts \eqref{local:separation}.
All in all, the mapping $\pi^\flat_{\bm{a}}$ is upper semicontinuous.

\medskip

For the second assertion, let $K\in\mc{K}^d$ and $\eps>0$,
and select $\delta>0$ according to the first assertion.
Let $L\in\mc{K}^d$ with $\dist_H(K,L)\le\delta$ and
$L'\in\Pi^\flat_{\bm{a}}(L)$.
Then there exists $\bm{x}\in\pi^\flat_{\bm{a}}(L)$ such that
$L'=\co\{\bm{x}\}$.
By the first assertion, and since $\pi^\flat_{\bm{a}}(K)$ is compact
by Theorem \ref{thm:basic:properties:projectors},
there exists $\bm{y}\in\pi^\flat_{\bm{a}}(K)$ such that
$\dist_{\mc{C}^\flat_{\bm{a}}}(\bm{x},\bm{y})\le\eps$.
Then $K':=\co\{\bm{y}\}\in\Pi^\flat_{\bm{a}}(K)$,
and Corollary \ref{cor:embedding:Lipschitz} yields
\[\dist_H(L',K')
=\dist_H(\co\{\bm{x}\},\co\{\bm{y}\})
\le\dist_{\mc{C}^\flat_{\bm{a}}}(\bm{x},\bm{y})
\le\eps,\]
which proves the second assertion.
\end{proof}

The mappings $\pi^\flat_{\bm a}$ and $\Pi^\flat_{\bm a}$ are in general
not lower semicontinuous.

\begin{example}\label{ex:not:lsc}
Let $a_1=\frac{1}{\sqrt{2}}(1,1)^T$,
$a_2=\frac{1}{\sqrt{2}}(-1,1)^T$ and
$a_3=(0,-1)^T$
and consider the sets
$K_k:=\co\{(-1,-1)^T,(-1,1)^T,(1,1)^T,(1,-1+1/k)^T\}$
and $K:=[-1,1]^2$.
Then $\dist_H(K_k,K)\to0$ and
\[\pi^\flat_{\bm a}(K_k)=\left\{\big((1,1)^T,(-1,1)^T,(-1,-1)^T\big)\right\},\]
while
$\pi^\flat_{\bm a}(K)=\left\{\big((1,1)^T,(-1,1)^T,(s,-1)^T\big):s\in[-1,1]\right\}$.
Moreover,
\[P:=\co\{(1,1)^T,(-1,1)^T,(1,-1)^T\}\in\Pi^\flat_{\bm a}(K),\]
while
$\Pi^\flat_{\bm a}(K_k)=\left\{\co\{(1,1)^T,(-1,1)^T,(-1,-1)^T\}\right\}$
for every $k$.
\end{example}

\begin{figure}
\centering
\def\epsfig{0.28}

\begin{tabular}{ccc}

\begin{tikzpicture}[scale=0.86]
\useasboundingbox (-1.6,-1.55) rectangle (1.6,1.55);

\coordinate (A) at (-1,-1);
\coordinate (B) at (1,{-1+\epsfig});
\coordinate (C) at (1,1);
\coordinate (D) at (-1,1);

\fill[blue!20] (A)--(B)--(C)--(D)--cycle;
\draw[thick] (A)--(B)--(C)--(D)--cycle;

\fill (A) circle (2pt);
\fill (B) circle (2pt);
\fill (C) circle (2pt);
\fill (D) circle (2pt);

\draw[->,thick] (C)--++(.55,.55) node[above right] {$a_1$};
\draw[->,thick] (D)--++(-.55,.55) node[above left] {$a_2$};
\draw[->,thick] (A)--++(0,-.75) node[below] {$a_3$};

\node at (0,0) {$K_k$};
\end{tikzpicture}

&{\hspace{1cm}}

\begin{tikzpicture}[scale=0.86]
\useasboundingbox (-1.6,-1.55) rectangle (1.6,1.55);

\coordinate (X1) at (1,1);
\coordinate (X2) at (-1,1);
\coordinate (X3) at (-1,-1);

\fill (X1) circle (2pt) node[below] {$[\pi^\flat_{\bm a}(K_k)]_1$};
\fill (X2) circle (2pt) node[below] {$[\pi^\flat_{\bm a}(K_k)]_2$};
\fill (X3) circle (2pt) node[above] {$[\pi^\flat_{\bm a}(K_k)]_3$};

\draw[->,thick] (X1)--++(.55,.55) node[above right] {$a_1$};
\draw[->,thick] (X2)--++(-.55,.55) node[above left] {$a_2$};
\draw[->,thick] (X3)--++(0,-.75) node[below] {$a_3$};

\end{tikzpicture}

&{\hspace{1cm}}

\begin{tikzpicture}[scale=0.86]
\useasboundingbox (-1.6,-1.55) rectangle (1.6,1.55);

\coordinate (X1) at (1,1);
\coordinate (X2) at (-1,1);
\coordinate (X3) at (-1,-1);

\fill[blue!20] (X1)--(X2)--(X3)--cycle;
\draw[thick] (X1)--(X2)--(X3)--cycle;

\fill (X1) circle (2pt);
\fill (X2) circle (2pt);
\fill (X3) circle (2pt);

\node at (0,-1.35) {$\Pi^\flat_{\bm a}(K_k)$};
\end{tikzpicture}

\\[1cm]

\begin{tikzpicture}[scale=0.86]
\useasboundingbox (-1.6,-1.55) rectangle (1.6,1.55);

\coordinate (A) at (-1,-1);
\coordinate (B) at (1,-1);
\coordinate (C) at (1,1);
\coordinate (D) at (-1,1);

\fill[blue!20] (A)--(B)--(C)--(D)--cycle;
\draw[thick] (A)--(B)--(C)--(D)--cycle;

\fill (A) circle (2pt);
\fill (B) circle (2pt);
\fill (C) circle (2pt);
\fill (D) circle (2pt);

\draw[->,thick] (C)--++(.55,.55) node[above right] {$a_1$};
\draw[->,thick] (D)--++(-.55,.55) node[above left] {$a_2$};
\draw[->,thick] (0,-1)--++(0,-.75) node[below] {$a_3$};

\node at (0,0) {$K$};
\end{tikzpicture}

&{\hspace{1cm}}

\begin{tikzpicture}[scale=0.86]
\useasboundingbox (-1.6,-1.55) rectangle (1.6,1.55);

\coordinate (X1) at (1,1);
\coordinate (X2) at (-1,1);
\coordinate (X3) at (-1,-1);
\coordinate (X4) at (1,-1);

\fill (X1) circle (2pt) node[below] {$[\pi^\flat_{\bm a}(K)]_1$};
\fill (X2) circle (2pt) node[below] {$[\pi^\flat_{\bm a}(K)]_2$};
\draw[very thick] (X3)--(X4)
  node[midway,above] {$[\pi^\flat_{\bm a}(K)]_3$}
  coordinate[midway] (M);

\draw[->,thick] (X1)--++(.55,.55) node[above right] {$a_1$};
\draw[->,thick] (X2)--++(-.55,.55) node[above left] {$a_2$};
\draw[->,thick] (M)--++(0,-.7) node[below] {$a_3$};

\end{tikzpicture}

&{\hspace{1cm}}

\begin{tikzpicture}[scale=0.86]
\useasboundingbox (-1.6,-1.55) rectangle (1.6,1.55);

\coordinate (X1) at (1,1);
\coordinate (X2) at (-1,1);
\coordinate (XL) at (-1,-1);
\coordinate (XM) at (0,-1);
\coordinate (XR) at (1,-1);

\fill[blue!10] (X2)--(X1)--(XL)--cycle;
\fill[blue!10] (X2)--(X1)--(XM)--cycle;
\fill[blue!10] (X2)--(X1)--(XR)--cycle;

\draw[dashed] (X2)--(XL)--(X1);
\draw[dashed] (X2)--(XM)--(X1);
\draw[thick]  (X2)--(XR)--(X1)--cycle;

\fill (X1) circle (2pt);
\fill (X2) circle (2pt);
\fill (XL) circle (2pt);
\fill (XM) circle (2pt);
\fill (XR) circle (2pt);

\node at (0,-1.35) {$\Pi^\flat_{\bm a}(K)$};
\end{tikzpicture}
\\[0.2cm]
\end{tabular}
\caption{In Example \ref{ex:not:lsc}, we have $\lim_{k\to\infty}\dist_H(K_k,K)=0$,
but $\pi^\flat_{\bm{a}}(K_k)$
and $\Pi^\flat_{\bm{a}}(K_k)$ are constant in $k$ with
$\pi^\flat_{\bm{a}}(K_k)\subsetneq\pi^\flat_{\bm{a}}(K)$
and $\Pi^\flat_{\bm{a}}(K_k)\subsetneq\Pi^\flat_{\bm{a}}(K)$.}
\label{fig:not:lsc}
\end{figure}

\section{Approximation and Galerkin sequences}
\label{sec:approximation}

We cite the following facts from Lemmas 1.8.12 and 1.8.14 in \cite{Schneider}.
Recall that the support function of a set $K\in\mc{K}^d$ is the function
\[\sigma(K,\,\cdot\,):\R^d\to\R,\quad\sigma(K,u):=\max_{x\in K}u^Tx.\]

\begin{lemma}
For $K,L\in\mc{K}^d$ and $u,v\in\R^d$ we have
\begin{align}
|\sigma(K,u)-\sigma(K,v)|&\le\|K\|\|u-v\|,\label{bound:support:function:vector}\\
\dist_H(K,L)&=\sup_{u\in S^{d-1}}|\sigma(K,u)-\sigma(L,u)|.\label{bound:distance}
\end{align}
\end{lemma}

Recalling the chordal covering radius
\[\delta_{\bm{a}}:=\max_{u\in S^{d-1}}\min_{i\in\{1,\ldots,N\}}\|u-a_i\|\]
we have the following approximation result.

\begin{theorem}\label{thm:approximation}
For all $K\in\mc{K}^d$ and $P\in\Pi^\flat_{\bm{a}}(K)$,
we have
\[\dist(P,K)=0
\quad\text{and}\quad
\dist(K,P)\le2\delta_{\bm{a}}\|K\|.\]
\end{theorem}

\begin{proof}
Lemma \ref{lemma:subset} yields $\dist(P,K)=0$ and $\|P\|\le\|K\|$.
Let $a\in S^{d-1}$.
Then there exists $i$ such that $\|a-a_i\|\le\delta_{\bm{a}}$.
By definition of $\Pi^\flat_{\bm{a}}$ we have
\begin{equation}\label{exact:at:support:point}
\sigma(K,a_i)=\sigma(P,a_i).
\end{equation}
Using \eqref{bound:support:function:vector}, \eqref{exact:at:support:point}
and again \eqref{bound:support:function:vector} we obtain
\begin{align*}
&|\sigma(K,a)-\sigma(P,a)|\\
&\le|\sigma(K,a)-\sigma(K,a_i)|
+|\sigma(K,a_i)-\sigma(P,a_i)|
+|\sigma(P,a_i)-\sigma(P,a)|\\
&\le\|K\|\|a-a_i\|+0+\|P\|\|a_i-a\|
\le2\|K\|\|a-a_i\|
\le2\delta_{\bm{a}}\|K\|.
\end{align*}
Since $a$ was arbitrary, identity \eqref{bound:distance} yields the desired result.
\end{proof}

We introduce nested Galerkin sequences.

\begin{definition}
Let $(N_k)_k\subset\N$ be strictly increasing, and for $k\in\N$ let
$\bm{a}_k\in\R^{d\times N_k}$
with $\|[\bm{a}_k]_i\|=1$ for all $i$
and $[\bm{a}_k]_i\neq[\bm{a}_k]_j$ for $i\neq j$.
\begin{itemize}
\item [a)] We call $(\bm{a}_k)_k$ a Galerkin sequence if
$\lim_{k\to\infty}\delta_{\bm{a}_k}=0$.
\item [b)] We call $(\bm{a}_k)_k$ nested if $[\bm{a}_k]_i=[\bm{a}_\ell]_i$ for all $k\le\ell$
and $i\in\{1,\ldots,N_k\}$.
\end{itemize}
\end{definition}

We present an example of a nested Galerkin sequence.
It has many symmetries, which is helpful for theoretical considerations,
and it is easy to handle in computations.

\begin{example}\label{ex:spherical}
Let $d\ge 3$ and let $s:[0,\pi]^{d-2}\times[0,2\pi]\to S^{d-1}$
be the spherical-coordinate parametrization
\[s(\phi):=
(\cos\phi_1, \sin\phi_1\cos\phi_2,\ldots,
(\prod_{j=1}^{d-2}\sin\phi_j)\cos\phi_{d-1},
(\prod_{j=1}^{d-2}\sin\phi_j)\sin\phi_{d-1}).\]
For $k\in\mathbb N$, define
\[G_k:=\{\tfrac{m\pi}{2^k}:m=0,\ldots,2^k\},\quad
H_k:=\{\tfrac{m\pi}{2^k}:m=0,\ldots,2^{k+1}\},\]
\[A_k:=\left\{s(\phi_1,\ldots,\phi_{d-1}):
\phi_i\in G_k\ \text{for}\ i=1,\ldots,d-2,\ \phi_{d-1}\in H_k\right\}.\]
Let $N_k:=\#A_k$ and choose
$\bm a_k=\bigl([\bm a_k]_1,\ldots,[\bm a_k]_{N_k}\bigr)\in\mathbb R^{d\times N_k}$
such that $\{\bm a_k\}=A_k$.
Since $A_k\subset A_{k+1}$, the vectors $\bm a_k$ may be indexed
recursively such that $[\bm a_{k+1}]_i=[\bm a_k]_i$ for $i=1,\ldots,N_k$.

\medskip

We show that $\lim_{k\to\infty}\delta_{\bm a_k}=0$.
Let $u\in S^{d-1}$ and $\psi\in[0,\pi]^{d-2}\times[0,2\pi]$ with $u=s(\psi)$.
There exists $\phi=(\phi_1,\ldots,\phi_{d-1})\in G_k^{d-2}\times H_k$
such that
\[|\psi_i-\phi_i|\le\tfrac{\pi}{2^{k+1}}
\quad\text{for}\quad i=1,\ldots,d-1.\]
Clearly $s(\phi)\in A_k$.
Since $s\in C^1([0,\pi]^{d-2}\times[0,2\pi],\R^d)$
and $[0,\pi]^{d-2}\times[0,2\pi]$ is compact,
there exists $L>0$ such that $s$ is $L$-Lipschitz, and
\[\|u-s(\phi)\|=\|s(\psi)-s(\phi)\|\le L\|\psi-\phi\|\le L\pi\sqrt{d-1}/2^{k+1}.\]
Consequently,
\[\lim_{k\to\infty}\delta_{\bm a_k}
\le\lim_{k\to\infty}L\pi\sqrt{d-1}/2^{k+1}
=0,\]
which proves that $(\bm a_k)_k$ is a nested Galerkin sequence.
\end{example}

It follows immediately from Theorem \ref{thm:approximation}
that Galerkin sequences approximate convex bodies locally uniformly to
arbitrary precision.

\begin{corollary}
If $(\bm{a}_k)_k$ is a Galerkin sequence, then
\[\inf_{P\in\mc{P}^\flat_{\bm{a}_k}}\dist_H(P,K)\le2\delta_{\bm{a}_k}\|K\|\quad
\forall\,K\in\mc{K}^d.\]
\end{corollary}

Nested sequences induce nested spaces.

\begin{proposition}
If $(\bm{a}_k)_k$ is nested, then the following holds for all $k$.
\begin{itemize}
\item [a)] For every $P_k\in\mc{P}^\flat_{\bm{a}_k}$
we have $P_k\in\Pi^\flat_{\bm{a}_{k+1}}(P_k)$.
\item [b)] In particular we have $\mc{P}^\flat_{\bm{a}_k}\subset\mc{P}^\flat_{\bm{a}_{k+1}}$.
\item [c)] For every $P_k\in\mc{P}^\flat_{\bm{a}_k}$
we have $\pi^\flat_{\bm{a}_{k+1}}(P_k)\cap\partial\mc{C}^\flat_{\bm{a}_{k+1}}\neq\emptyset$.
\end{itemize}
\end{proposition}

\begin{proof}
Let $P_k\in\mc{P}^\flat_{\bm{a}_k}$.
Then there exists $\bm{x}_k\in\mc{C}^\flat_{\bm{a}_k}$ with $P_k=\co\{\bm{x}_k\}$.
By part a) of Theorem \ref{construction:site},
we have $\bm{x}_k\in\pi^\flat_{\bm{a}_k}(\co\{\bm{x}_k\})$,
and since $[\bm{a}_{k+1}]_i=[\bm{a}_k]_i$ for $i\in\{1,\ldots,N_k\}$,
this means that
\begin{equation}\label{local:project:back}
[\bm{x}_k]_i\in\argmax_{x\in\co\{\bm{x}_k\}}[\bm{a}_{k+1}]_i^Tx
\quad\forall\,i\in\{1,\ldots,N_k\}.
\end{equation}
In addition, we select
\begin{equation}\label{local:extend}
x_i\in\argmax_{x\in\co\{\bm{x}_k\}}[\bm{a}_{k+1}]_i^Tx
\quad\forall\,i\in\{N_k+1,\ldots,N_{k+1}\}
\end{equation}
and define
\begin{equation}\label{loc:construct:x_k+1}
\bm{x}_{k+1}:=([\bm{x}_k]_1,\ldots,[\bm{x}_k]_{N_k},x_{N_k+1},\ldots,x_{N_{k+1}}).
\end{equation}
By \eqref{local:project:back}, \eqref{local:extend} and \eqref{loc:construct:x_k+1},
we have $\{\bm{x}_k\}\subset\{\bm{x}_{k+1}\}\subset\co\{\bm{x}_k\}$,
and since $\co$ 
is monotone w.r.t.\ 
inclusion, we obtain
$\co\{\bm{x}_k\}\subset\co\{\bm{x}_{k+1}\}\subset\co\{\bm{x}_k\}$
and hence
\begin{equation}\label{loc:P_k:from:x_k+1}
P_k=\co\{\bm{x}_k\}=\co\{\bm{x}_{k+1}\}.
\end{equation}
Furthermore, statements \eqref{local:project:back},
\eqref{local:extend} and \eqref{loc:construct:x_k+1}
show $\bm{x}_{k+1}\in\pi^\flat_{\bm{a}_{k+1}}(P_k)$,
which implies $\bm{x}_{k+1}\in\mc{C}^\flat_{\bm{a}_{k+1}}$.
With \eqref{loc:P_k:from:x_k+1} we obtain
$P_k=\co\{\bm{x}_{k+1}\}\in\Pi^\flat_{\bm{a}_{k+1}}(P_k)$,
so a) holds.
Statement b) is an immediate consequence of a).

\medskip

We continue the above argument.
By \eqref{local:extend} for every $i\in\{N_k+1,\ldots,N_{k+1}\}$
there exists $j\in\{1,\ldots,N_k\}$ with
\[[\bm{a}_{k+1}]_i^T[\bm{x}_{k+1}]_j=[\bm{a}_{k+1}]_i^T[\bm{x}_k]_j
=[\bm{a}_{k+1}]_i^Tx_i=[\bm{a}_{k+1}]_i^T[\bm{x}_{k+1}]_i.\]
Hence by Theorem \ref{theorem:int:CA},
we have $\bm{x}_{k+1}\in\partial\mc{C}^\flat_{\bm{a}_{k+1}}$.
\end{proof}

If a Galerkin sequence is nested, then approximations from the Galerkin polytope spaces
can be chosen in a nested way without sacrificing accuracy.

\begin{lemma}
If $(\bm{a}_k)_k$ is a nested Galerkin sequence, then for any $K\in\mc{K}^d$,
there exist $\bm{x}_k\in\pi^\flat_{\bm{a}_k}(K)$ such that for all $k\in\N$ we have
\begin{equation}\label{nested:xk}
\bm{x}_{k+1}=([\bm{x}_k]_1,\ldots,[\bm{x}_k]_{N_k},x_{N_k+1},\ldots,x_{N_{k+1}})
\end{equation}
and $P_k:=\co\{\bm{x}_k\}$
satisfy $P_k\in\Pi^\flat_{\bm{a}_k}(K)$,
$P_k\subset P_{k+1}\subset K$ and
\begin{equation}\label{loc:estimate}
\dist_H(P_k,K)\le2\delta_{\bm{a}_k}\|K\|.
\end{equation}
\end{lemma}

\begin{proof}
We construct the sequences ($\bm{x}_k)_k$ and $(P_k)_k$ recursively.
Take any $\bm{x}_0\in\pi^\flat_{\bm{a}_0}(K)$.
Then $P_0=\co\{\bm{x}_0\}\in\Pi^\flat_{\bm{a}_0}(K)$.
By Theorem \ref{thm:approximation} the set $P_0$ satisfies \eqref{loc:estimate} with $k=0$.

Now assume that $(\bm{x}_j)_{j=0}^k$ and $(P_j)_{j=0}^k$ with the required properties
have already been constructed.
For every $i\in\{N_k+1,\ldots,N_{k+1}\}$ take arbitrary
$x_i\in\argmax_{x\in K}[\bm{a}_{k+1}]_i^Tx$
to complete \eqref{nested:xk}.
Then $\bm{x}_{k+1}\in\pi^\flat_{\bm{a}_{k+1}}(K)$.
Because of $\{\bm{x}_k\}\subset\{\bm{x}_{k+1}\}\subset K$,
it follows that $\co\{\bm{x}_k\}\subset\co\{\bm{x}_{k+1}\}\subset K$
and hence $P_k\subset P_{k+1}\subset K$, as desired.
Since $P_{k+1}\in\Pi^\flat_{\bm{a}_{k+1}}(K)$,
Theorem \ref{thm:approximation} applies and yields \eqref{loc:estimate} with $k+1$
in lieu of $k$.
\end{proof}

\section{Global optimization in $\mc{K}^d$}
\label{sec:optimization}

We equip the space
\[\mc{B}:=\{\mc{M}\subset\mc{K}^d:\mc{M}\neq\emptyset,\ \sup_{K\in\mc{M}}\dist_H(K,\{0\})<\infty\}\]
of nonempty bounded subsets of $(\mc{K}^d,\dist_H)$
with the Hausdorff semi-distance and the Hausdorff-distance given by
\begin{align*}
&\mc{D}:\mc{B}\times\mc{B}\to\R_+,\quad
\mc{D}(\mc{M},\tilde{\mc{M}})
:=\sup_{K\in\mc{M}}\inf_{\tilde{K}\in\tilde{\mc{M}}}\dist_H(K,\tilde{K}),\\
&\mc{D}_H:\mc{B}\times\mc{B}\to\R_+,\
\mc{D}_H(\mc{M},\tilde{\mc{M}})
:=\max\{\mc{D}(\mc{M},\tilde{\mc{M}}),\mc{D}(\tilde{\mc{M}},\mc{M})\}.
\end{align*}

The following result is \cite[Theorem 44]{Rieger:Galerkin}.
A related statement can e.g.\ be found in \cite[Theorem 2]{Antunes}.
It states that under generic conditions, the global minimizers of approximate
optimization problems converge to global minimizers of the original problem.

\begin{theorem} \label{minconvthm}
Let $\mc{M}\in\mc{B}$ be compact,
and let $(\mc{M}_k)_{k=0}^\infty$ be a sequence of nonempty compact subsets
$\mc{M}_k\in\mc{B}$ with
\begin{align} \label{setconv}
\lim_{k\to\infty}\mc{D}_H(\mc{M},\mc{M}_k)=0.
\end{align}
Let $\Phi:\mc{K}^d\to\R$ be continuous, let $(\Phi_k)_{k=0}^\infty$
with $\Phi_k:\mc{M}_k\to\R$ be a sequence of mappings satisfying
\begin{align} \label{uniapprox}
\lim_{k\to\infty}\sup_{K\in\mc{M}_k}|\Phi(K)-\Phi_k(K)|=0,
\end{align}
and let $\argmin_{K\in\mc{M}_k}\Phi_k(K)\neq\emptyset$ for all $k$.
Then $\argmin_{K\in\mc{M}}\Phi(K)\neq\emptyset$, and
\begin{equation} \label{desiredconv}
\lim_{k\to\infty}\mc{D}(\argmin_{K\in\mc{M}_k}\Phi_k(K),
\argmin_{K\in\mc{M}}\Phi(K))=0.
\end{equation}
\end{theorem}

In the following we analyze a model problem.
Let $\Phi:(\mc{K}^d,\dist_H)\to\R$ be continuous
and let $\Psi:(\mc{K}^d,\dist_H)\to(\R^m,\|\cdot\|_\infty)$ be Lipschitz on bounded
subcollections of $\mc{K}^d$.
Let $\hat{K}\in\mc{K}^d$, assume that
\[\mc{M}:=\{K\in\mc{K}^d:\Psi(K)\le 0,\ K\subset\hat{K}\}\neq\emptyset,\]
where the inequality is understood componentwise, and consider the problem
\begin{equation} \label{aop}
\min_{K\in\mc{K}^d}\Phi(K)\quad\text{subject to}\quad \Psi(K)\le 0,\ K\subset\hat{K}.
\end{equation}

We recall Blaschke's selection theorem, which is \cite[Theorem 1.8.7]{Schneider}.

\begin{theorem}\label{Blaschke}
Let $\{K_k\}_{k\in\N}\subset\mc{K}^d$
and $R>0$ such that $\|K_k\|\le R$ for all $k$.
Then there exist a subsequence $\N'\subset\N$ and $K^*\in\mc{K}^d$ with
\[\lim_{\N'\ni k\to\infty}\dist_H(K_k,K^*)=0.\]
\end{theorem}

The condition $\|K_k\|\le R$ for all $k$ is equivalent with the existence
of a set $K\in\mc{K}^d$ such that $K_k\subset K$ for all $k$.

\begin{lemma}\label{lem:m:compact}
The collection $\mc{M}$ is sequentially compact.
\end{lemma}

\begin{proof}
Let $\{K_k\}_{k\in\N}\subset\mc{M}$.
By Blaschke's selection theorem, there exist $\N'\subset\N$ and $K^*\in\mc{K}^d$
such that
\[\lim_{\N'\ni k\to\infty}\dist_H(K_k,K^*)=0.\]
Since $\Psi$ is continuous, we have $\Psi(K^*)\le0$.
Since $K_k\subset\hat{K}$, we also have $\dist(K_k,\hat{K})=0$.
Hence for all $k$ we obtain
\[\dist(K^*,\hat{K})
\le\dist(K^*,K_k)+\dist(K_k,\hat{K})
\le\dist_H(K^*,K_k),\]
which implies $\dist(K^*,\hat{K})=0$ and hence $K^*\subset\hat{K}$.
\end{proof}

To approximate solutions of \eqref{aop} we fix a nested Galerkin sequence $(\bm{a}_k)$
with $\bm{a}_k\in\R^{d\times N_k}$ for all $k$.
We will need an elementary technical result.

\begin{figure}
\centering

\begin{tikzpicture}[scale=0.8]
\useasboundingbox (-2.4,-2.15) rectangle (2.7,2.15);

\def\RL{2}
\def\RK{1.4}
\def\rr{0.6}
\def\cL{1.15}
\def\RLp{0.45}
\def\RBr{1.05} 

\coordinate (O) at (0,0);
\coordinate (C) at (\cL,0);

\draw[thick] (O) circle (\RL);
\node[above left] at (-1.2,1.35) {$L$};

\fill[blue!10] (O) circle (\RK);
\draw[thick] (O) circle (\RK);
\node[left] at (-0.35,.8) {$K$};

\begin{scope}
  \clip (O) circle (\RK);
  \fill[blue!35] (C) circle (\RBr);
\end{scope}

\draw[dashed,thick] (C) circle (\RBr);
\node[above right] at (2,0.3) {$B_r(L')$};

\fill[blue!55] (C) circle (\RLp);
\draw[thick] (C) circle (\RLp);
\node at (C) {$L'$};

\draw[thick] (O) circle (\RK);

\node at (.55,.5) {$K'$};

\draw[|-|,thick] (-\RL,0)--(-\RK,0)
  node[midway,above] {$r$};

\draw[|-|,thick]
  ({\cL+\RLp},0)--({\cL+\RBr},0)
  node[midway,below] {$r$};

\begin{scope}
  \clip (O) circle (\RK);
  \draw[red,thick] (C) circle (\RBr);
\end{scope}

\begin{scope}
  \clip (C) circle (\RBr);
  \draw[red,thick] (O) circle (\RK);
\end{scope}

\end{tikzpicture}
\caption{Depicts the situation from Lemma \ref{lem:elementary}.
The set $K'$ is the region with red border.}
\end{figure}
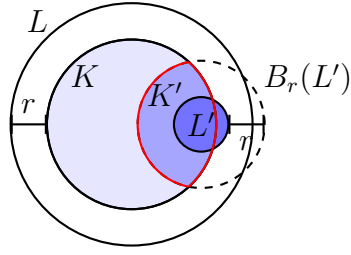

\begin{lemma}\label{lem:elementary}
Let $K,L\in\mc{K}^d$ with $K\subset L$ and denote $r:=\dist(L,K)$.
Then for every $L'\in\mc{K}^d$ with $L'\subset L$
the set $K':=B_r(L')\cap K$ satisfies
\[K'\in\mc{K}^d,\quad K'\subset K\quad\text{and}\quad\dist_H(K',L')\le\dist(L,K).\]
\end{lemma}

\begin{proof}
For every $y\in L'$, we have $y\in L$,
and since $K\in\mc{K}^d$ there exists $x\in K$ with $\|x-y\|\le r$.
This shows both $K'\neq\emptyset$ and $\dist(L',K')\le r$.
Conversely, for every $x\in K'$, by definition, there is $y\in L'$ with $x\in B_r(y)$.
Hence $\dist(K',L')\le r$.
Since $K,B_r(L')\in\mc{K}^d$, so is their intersection $K'$.
\end{proof}

Now we establish a first result concerning three finitely parameterized
approximations of the collection $\{K\in\mc{K}^d:K\subset\hat{K}\}$.
Depending on the application, one may be more natural than the others.
Recall the projectors $\pi^\sharp_{\bm{a}}$ and $\Pi^\sharp_{\bm{a}}$
from \eqref{sharp:1} and \eqref{sharp:2}.

\begin{lemma}\label{lem:sets:only}
Let $\hat{P}_k\in\Pi^\flat_{\bm{a}_k}(\hat{K})$.
Then the estimates
\begin{align}
&\mc{D}(\{K\in\mc{K}^d:K\subset\hat{K}\},\{P\in\mc{P}^\flat_{\bm{a}_k}:P\subset\hat{P}_k\})\le4\delta_{\bm{a}_k}\|\hat{K}\|,
\label{a1}\\
&\mc{D}(\{P\in\mc{P}^\flat_{\bm{a}_k}:P\subset\hat{P}_k\},
\{K\in\mc{K}^d:K\subset\hat{K}\})\le0,\label{a2}\\
&\mc{D}(\{K\in\mc{K}^d:K\subset\hat{K}\},
\{P\in\mc{P}^\flat_{\bm{a}_k}:P\subset\hat{K}\})\le2\delta_{\bm{a}_k}\|\hat{K}\|,\label{b1}\\
&\mc{D}(\{P\in\mc{P}^\flat_{\bm{a}_k}:P\subset\hat{K}\},
\{K\in\mc{K}^d:K\subset\hat{K}\})\le0,\label{b2}\\
&\mc{D}(\{K\in\mc{K}^d:K\subset\hat{K}\},
\{P\in\mc{P}^\flat_{\bm{a}_k}:P\subset\Pi^\sharp_{\bm{a}_k}(\hat{K})\})\le2\delta_{\bm{a}_k}\|\hat{K}\|,\label{c1}\\
&\mc{D}(\{P\in\mc{P}^\flat_{\bm{a}_k}:P\subset\Pi^\sharp_{\bm{a}_k}(\hat{K})\},
\{K\in\mc{K}^d:K\subset\hat{K}\})\le\frac{2-\delta_{\bm{a}_k}}{1-\delta_{\bm{a}_k}}\delta_{\bm{a}_k}\|\hat{K}\|\label{c2}
\end{align}
hold, where for estimate \eqref{c2} we require $\delta_{\bm{a}_k}\in(0,1)$.
\end{lemma}

\begin{proof}
By Lemma \ref{lemma:subset} and Theorem \ref{thm:approximation} 
all $K\in\mc{K}^d$ and $P\in\Pi^\flat_{\bm{a}}(K)$ satisfy
\begin{equation}\label{lemthm}
P\subset K,\ \dist(P,K)=0,\ \text{and}\ \dist(K,P)\le2\delta_{\bm{a}}\|K\|.
\end{equation}
We will occasionally use these facts without mention.

\medskip

We have $\hat{P}_k\subset\hat{K}$.
Let $K\in\mc{K}^d$ with $K\subset\hat{K}$.
Applying Lemma \ref{lem:elementary} with $(\hat{P}_k,\hat{K},K)$ in lieu of $(K,L,L')$
yields $L\in\mc{K}^d$ (in lieu of $K'$) with $L\subset\hat{P}_k$ and
\[\dist_H(K,L)\le\dist(\hat{K},\hat{P}_k)\le2\delta_{\bm{a}_k}\|\hat{K}\|.\]
But for any $P\in\Pi^\flat_{\bm{a}_k}(L)$ we have $P\in\mc{P}^\flat_{\bm{a}_k}$
as well as $P\subset L\subset\hat{P}_k$ and
\[\dist_H(L,P)\le2\delta_{\bm{a}_k}\|L\|\le2\delta_{\bm{a}_k}\|\hat{K}\|.\]
Combining above estimates yields $\dist_H(K,P)\le4\delta_{\bm{a}_k}\|\hat{K}\|$
and hence \eqref{a1}.
Since every $P\in\mc{P}^\flat_{\bm{a}_k}$ with $P\subset\hat{P}_k$
satisfies $P\in\mc{K}^d$ and $P\subset\hat{K}$, inequality \eqref{a2} holds.

\medskip

Estimate \eqref{b1} follows directly from \eqref{lemthm} for any
$P\in\Pi^\flat_{\bm{a}_k}(K)$, and inequality \eqref{b2} is trivial.

\medskip

Let $K\in\mc{K}^d$ with $K\subset\hat{K}$, and let $P\in\Pi^\flat_{\bm{a}_k}(K)$.
Then $P\in\mc{P}^\flat_{\bm{a}_k}$, by \eqref{lemthm} and Theorem \ref{oldapprox}
we have $P\subset K\subset\hat{K}\subset\Pi^\sharp_{\bm{a}_k}(\hat{K})$,
and by \eqref{lemthm} we have $\dist_H(K,P)\le2\delta_{\bm{a}_k}\|K\|\le2\delta_{\bm{a}_k}\|\hat{K}\|$,
which shows \eqref{c1}.

\medskip

Let $P\in\mc{P}^\flat_{\bm{a}_k}$ with $P\subset\Pi^\sharp_{\bm{a}_k}(\hat{K})$.
By Theorem \ref{oldapprox} we have $\hat{K}\subset\Pi^\sharp_{\bm{a}_k}(\hat{K})$.
Hence by Lemma \ref{lem:elementary} with $(\hat{K},\Pi^\sharp_{\bm{a}_k}(\hat{K}),P)$
in lieu of $(K,L,L')$ and Theorem \ref{oldapprox},
there exists $K\in\mc{K}^d$ (in lieu of $K'$) with $K\subset\hat{K}$ and
\[\dist_H(P,K)\le\dist(\Pi^\sharp_{\bm{a}_k}(\hat{K}),\hat{K})
\le\frac{2-\delta_{\bm{a}_k}}{1-\delta_{\bm{a}_k}}\delta_{\bm{a}_k}\|\hat{K}\|.\]
Since $P$ was arbitrary, this shows \eqref{c2}.
\end{proof}

We establish that the constraint $\Psi$ is $\ell$-Lipschitz on the domain
of interest.

\begin{lemma}\label{lem:Psi:Lipschitz}
Assume that $\delta_{\bm{a}_k}\in(0,1)$ for all $k$, and let
$\hat{P}_k\in\Pi^\flat_{\bm{a}_k}(\hat{K})$ for all $k$.
Then there exists $\ell>0$ such that $\Psi$ is $\ell$-Lipschitz on
the union of $\mc{M}$ with
\[\cup_{k\in\N}
(\{P\in\mc{P}^\flat_{\bm{a}_k}:P\subset\hat{P}_k\}
\cup\{P\in\mc{P}^\flat_{\bm{a}_k}:P\subset\hat{K}\}
\cup\{P\in\mc{P}^\flat_{\bm{a}_k}:P\subset\Pi^\sharp_{\bm{a}_k}(\hat{K})\}).\]
\end{lemma}

\begin{proof}
Since $\mc{M}\subset\{K\in\mc{K}^d:K\subset\hat{K}\}$ and
$\hat{P}_k\subset\hat{K}$ for all $k$, both $\mc{M}$ and the collection
$\cup_{k\in\N}(\{P\in\mc{P}^\flat_{\bm{a}_k}:P\subset\hat{P}_k\}
\cup\{P\in\mc{P}^\flat_{\bm{a}_k}:P\subset\hat{K}\})$
are bounded.
Moreover, since $(\bm{a}_k)_k$ is a Galerkin sequence
and $\delta_{\bm{a}_k}\in(0,1)$ for all $k$, we have
\[\sup_{k\in\N}
\frac{2-\delta_{\bm{a}_k}}{1-\delta_{\bm{a}_k}}
\delta_{\bm{a}_k}<\infty.\]
Thus Theorem \ref{oldapprox} implies that the sequence
$(\Pi^\sharp_{\bm{a}_k}(\hat{K}))_k$ is uniformly bounded.
Consequently,
$\bigcup_{k\in\N}\{P\in\mc{P}^\flat_{\bm{a}_k}:P\subset\Pi^\sharp_{\bm{a}_k}(\hat{K})\}$
is bounded as well.
Since $\Psi$ is Lipschitz on bounded subcollections of $\mc{K}^d$,
the assertion ensues.
\end{proof}

Now we investigate the feasible sets of approximate problems.

\begin{lemma}\label{lem:def:Psi}
Assume that $\delta_{\bm{a}_k}\in(0,1)$ for all $k$.
Let $\hat{P}_k\in\Pi^\flat_{\bm{a}_k}(\hat{K})$ for all $k$,
and let $\ell>0$ be as in Lemma \ref{lem:Psi:Lipschitz}.
Let $\mathbbm{1}\in\R^m$ be the vector of ones, define
\begin{align*}
&\Psi^{\flat\flat}_k(P):=\Psi(P)-4\ell\delta_{\bm{a}_k}\|\hat{K}\|\mathbbm{1},\\
&\Psi^{\flat}_k(P):=\Psi(P)-2\ell\delta_{\bm{a}_k}\|\hat{K}\|\mathbbm{1},\\
&\Psi^{\flat\sharp}_k(P):=\Psi(P)-2\ell\delta_{\bm{a}_k}\|\hat{K}\|\mathbbm{1},
\end{align*}
and define
\begin{align*}
&\mc{M}^{\flat\flat}_k:=\{P\in\mc{P}^\flat_{\bm{a}_k}:\Psi^{\flat\flat}_k(P)\le0,\
P\subset\hat{P}_k\},\\
&\mc{M}^{\flat}_k:=\{P\in\mc{P}^\flat_{\bm{a}_k}:\Psi^{\flat}_k(P)\le0,\
P\subset\hat{K}\},\\
&\mc{M}^{\flat\sharp}_k:=\{P\in\mc{P}^\flat_{\bm{a}_k}:\Psi^{\flat\sharp}_k(P)\le0,\
P\subset\Pi^\sharp_{\bm{a}_k}(\hat{K})\}.
\end{align*}
Then $\mc{M}^{\flat\flat}_k\neq\emptyset$, $\mc{M}^{\flat}_k\neq\emptyset$ and $\mc{M}^{\flat\sharp}_k\neq\emptyset$ for all $k$,
$\mc{M}^{\flat\flat}_k$, $\mc{M}^{\flat}_k$ and $\mc{M}^{\flat\sharp}_k$
are sequentially compact w.r.t.\ $\dist_H$ for all $k$,
and we have
\[\lim_{k\to\infty}\mc{D}_H(\mc{M},\mc{M}^{\flat\flat}_k)
=\lim_{k\to\infty}\mc{D}_H(\mc{M},\mc{M}^\flat_k)
=\lim_{k\to\infty}\mc{D}_H(\mc{M},\mc{M}^{\flat\sharp}_k)=0.\]
\end{lemma}

\begin{proof}
By \eqref{a2}, \eqref{b2} and \eqref{c2}
the collections $\mc{M}^{\flat\flat}_k$, $\mc{M}^{\flat}_k$ and $\mc{M}^{\flat\sharp}_k$
are bounded.
Hence Blaschke's selection theorem yields that they are sequentially pre-compact
w.r.t.\ $\dist_H$.
They are also sequentially closed, because
$\Psi^{\flat\flat}_k$, $\Psi^{\flat}_k$ and $\Psi^{\flat\sharp}_k$ are continuous,
the collection $\mc{P}^\flat_{\bm{a}_k}$ is closed by Theorem \ref{thm:P:closed},
and $\hat{P}_k\in\mc{K}^d$, $\hat{K}\in\mc{K}^d$ and
$\Pi^\sharp_{\bm{a}_k}(\hat{K})\in\mc{K}^d$ for all $k$.

\medskip

Let $K\in\mc{M}$.
Then $K\subset\hat{K}$ and $\Psi(K)\le 0$.
By \eqref{a1} there exists $P\in\mc{P}^\flat_{\bm{a}_k}$ with $P\subset\hat{P}_k$
such that $\dist_H(K,P)\le4\delta_{\bm{a}_k}\|\hat{K}\|$.
For $j\in\{1,\ldots,m\}$ we estimate
\[\Psi_j(P)
\le|\Psi_j(P)-\Psi_j(K)|+\Psi_j(K)
\le\ell\dist_H(P,K)\le4\ell\delta_{\bm{a}_k}\|\hat{K}\|.\]
In particular, we have $\Psi_k^{\flat\flat}(P)\le0$ and hence $P\in\mc{M}^{\flat\flat}_k$
and $\mc{M}^{\flat\flat}_k\neq\emptyset$.
Since $K$ was arbitrary, this shows
$\mc{D}(\mc{M},\mc{M}^{\flat\flat}_k)\le4\delta_{\bm{a}_k}\|\hat{K}\|\to0$ as $k\to\infty$.

\medskip

Assume that $\lim_{k\to\infty}\mc{D}(\mc{M}^{\flat\flat}_k,\mc{M})=0$ is false.
Then there exist $\eps>0$, a subsequence $\N'\subset\N$ and sets $P_k\in\mc{M}^{\flat\flat}_k$ for $k\in\N'$ with
\begin{equation}\label{loc:2}
\mc{D}(P_k,\mc{M})\ge\eps.
\end{equation}
Since the $\mc{M}^{\flat\flat}_k$ are uniformly bounded, there exists $c>0$
with $\|P_k\|\le c$ for all $k\in\N'$.
By Blaschke's selection theorem there exist a subsequence $\N''\subset\N'$
and $K^*\in\mc{K}^d$ with
\begin{equation}\label{loc:1}
\lim_{\N''\ni k\to\infty}\dist_H(P_k,K^*)=0.
\end{equation}
Since $P_k\in\mc{M}^{\flat\flat}_k$ for all $k$,
we have $\Psi^{\flat\flat}_k(P_k)\le 0$
and hence $\Psi(P_k)\le 4\ell\delta_{\bm{a}_k}\|\hat{K}\|\mathbbm{1}$.
By continuity of $\Psi$ and since $\lim_{k\to\infty}\delta_{\bm{a}_k}=0$,
we obtain $\Psi(K^*)\le0$.
By definition of $\mc{M}^{\flat\flat}_k$ we have $P_k\subset\hat{P}_k\subset\hat{K}$,
and hence \eqref{loc:1} also implies $K^*\subset\hat{K}$.
We have shown that $K^*\in\mc{M}$, which together with \eqref{loc:1} contradicts \eqref{loc:2}.
Hence our initial assumption was false, and we have
$\lim_{k\to\infty}\mc{D}(\mc{M}^{\flat\flat}_k,\mc{M})=0$.

\medskip

The proofs of the remaining statements are similar.
\end{proof}

Now we gather the results from this section in a final statement.
We use the abbreviation \emph{lsc} for \emph{lower semicontinuous}.

\begin{theorem} \label{concreteconv1}
For all $k\in\N$, assume that $\delta_{\bm{a}_k}\in(0,1)$,
and let $\Psi^{\flat\flat}_k$, $\Psi^{\flat}_k$ and $\Psi^{\flat\sharp}_k$
be as in Lemma \ref{lem:def:Psi}.
\begin{itemize}
\item  [a)] If $\Phi^{\flat\flat}_k:\mc{K}^d\to\R$ is lsc
and $\lim_{k\to\infty}\sup_{K\in\mc{M}^{\flat\flat}_k}|\Phi(K)-\Phi^{\flat\flat}_k(K)|=0$,
then $\argmin_{K\in\mc{M}}\Phi(K)\neq\emptyset$ and
\[\lim_{k\to\infty}\mc{D}(\argmin_{K\in\mc{M}^{\flat\flat}_k}
\Phi^{\flat\flat}_k(K),\argmin_{K\in\mc{M}}\Phi(K))=0.\]
\item  [b)] If $\Phi^{\flat}_k:\mc{K}^d\to\R$ is lsc
and $\lim_{k\to\infty}\sup_{K\in\mc{M}^{\flat}_k}|\Phi(K)-\Phi^{\flat}_k(K)|=0$,
then $\argmin_{K\in\mc{M}}\Phi(K)\neq\emptyset$ and
\[\lim_{k\to\infty}\mc{D}(\argmin_{K\in\mc{M}^{\flat}_k}
\Phi^{\flat}_k(K),\argmin_{K\in\mc{M}}\Phi(K))=0.\]
\item  [c)] If $\Phi^{\flat\sharp}_k:\mc{K}^d\to\R$ is lsc
and $\lim_{k\to\infty}\sup_{K\in\mc{M}^{\flat\sharp}_k}|\Phi(K)-\Phi^{\flat\sharp}_k(K)|=0$,
then $\argmin_{K\in\mc{M}}\Phi(K)\neq\emptyset$ and
\[\lim_{k\to\infty}\mc{D}(\argmin_{K\in\mc{M}^{\flat\sharp}_k}
\Phi^{\flat\sharp}_k(K),\argmin_{K\in\mc{M}}\Phi(K))=0.\]
\end{itemize}
\end{theorem}

\begin{proof}
By assumption and by Lemma \ref{lem:m:compact} the collection $\mc{M}$ is nonempty and compact.
By Lemma \ref{lem:def:Psi} the collections
$\mc{M}^{\flat\flat}_k$, $\mc{M}^{\flat}_k$ and $\mc{M}^{\flat\sharp}_k$
are nonempty and compact for all $k$.
Hence by lower semicontinuity, the sets of minimizers are nonempty in cases
a), b) and c).
In addition, Lemma \ref{lem:def:Psi} ensures that $\mc{M}^{\flat\flat}_k$,
$\mc{M}^{\flat}_k$ and $\mc{M}^{\flat\sharp}_k$ satisfy condition \eqref{setconv}.
Thus Theorem \ref{minconvthm} applies and yields the desired statement.
\end{proof}

\bibliographystyle{plain}
\bibliography{main}
\end{document}